\documentclass[a4paper,10pt]{amsart}

\usepackage{amsmath}
\usepackage{amsthm}
\usepackage{amssymb}
\usepackage{amsfonts}
\usepackage[shortlabels]{enumitem} 
\usepackage[normalem]{ulem}
\usepackage{cancel}
\usepackage{xcolor}
\usepackage[bookmarks=true,hyperindex,pdftex,colorlinks,citecolor=blue,urlcolor=cyan]{hyperref}
\usepackage{marginnote} 
\usepackage[all]{xy} 
\usepackage{mathrsfs}  

\DeclareMathOperator{\lspan}{span}                          
\DeclareMathOperator{\supp}{supp}                           
\DeclareMathOperator{\diam}{diam}                           
\DeclareMathOperator{\rad}{rad}                             
\DeclareMathOperator{\Lip}{Lip}                             
\DeclareMathOperator{\lip}{lip}                             
\DeclareMathOperator{\im}{im}                               

\newcommand{\NN}{\mathbb{N}}             
\newcommand{\ZZ}{\mathbb{Z}}             
\newcommand{\RR}{\mathbb{R}}             

\newcommand{\CC}{\mathcal{C}}             
\newcommand{\HH}{\mathcal{H}}             
\newcommand{\LL}{\mathcal{L}}            

\newcommand{\abs}[1]{\left|{#1}\right|}                     
\newcommand{\pare}[1]{\left({#1}\right)}                    
\newcommand{\set}[1]{\left\{{#1}\right\}}                   
\newcommand{\norm}[1]{\left\|{#1}\right\|}                  
\newcommand{\dual}[1]{{#1}^\ast}                            
\newcommand{\ball}[1]{B_{{#1}}}                             
\newcommand{\duality}[1]{\left<{#1}\right>}                 
\newcommand{\cl}[1]{\overline{#1}}                          
\newcommand{\weak}{\textit{w}}                              
\newcommand{\weaks}{\textit{w}$^\ast$}                      
\newcommand{\wsconv}{\stackrel{\dual{w}}{\longrightarrow}}  
\newcommand{\restrict}{\mathord{\upharpoonright}}           

\newcommand{\lipfree}[1]{\mathcal{F}({#1})}                 
\newcommand{\Free}{{\mathcal F}}
\newcommand{\lipnorm}[1]{\norm{#1}_L}                       
\newcommand{\mol}[1]{m_{#1}}                                

\newcommand{\eps}{\varepsilon}
\newcommand{\A}{\mathcal{A}}
\newcommand{\E}{\mathbb{E}}
\newcommand{\PP}{\mathbb{P}}

\theoremstyle{plain}
\newtheorem{theorem}{Theorem}[section]
\newtheorem{lemma}[theorem]{Lemma}
\newtheorem{corollary}[theorem]{Corollary}
\newtheorem{proposition}[theorem]{Proposition}

\newtheorem*{claim*}{Claim}

\newtheorem{maintheorem}{Theorem} 

\newtheorem*{informalthm}{Informal Theorem}

\makeatletter
\newtheorem*{rep@theorem}{\rep@title}
\newcommand{\newreptheorem}[2]{%
\newenvironment{rep#1}[1]{%
 \def\rep@title{#2 \ref{##1}}%
 \begin{rep@theorem}}%
 {\end{rep@theorem}}}
\makeatother

\newreptheorem{theorem}{Theorem}

\theoremstyle{definition}
\newtheorem*{definition*}{Definition}
\newtheorem{definition}[theorem]{Definition}
\newtheorem{example}[theorem]{Example}

\theoremstyle{remark}
\newtheorem{remark}[theorem]{Remark}

\begin{document}
\title[Purely 1-unrectifiable spaces and locally flat Lipschitz functions]{Purely 1-unrectifiable metric spaces and locally flat Lipschitz functions}

\author[R. J. Aliaga]{Ram\'on J. Aliaga}
\address[R. J. Aliaga]{Instituto Universitario de Matem\'atica Pura y Aplicada, Universitat Polit\`ecnica de Val\`encia, Camino de Vera S/N, 46022 Valencia, Spain}
\email{raalva@upvnet.upv.es}

\author[C. Gartland]{Chris Gartland}
\address[C. Gartland]{Texas A\&M University, College Station, TX 77843, USA}
\email{cgartland@math.tamu.edu}

\author[C. Petitjean]{Colin Petitjean}
\address[C. Petitjean]{LAMA, Univ Gustave Eiffel, UPEM, Univ Paris Est Creteil, CNRS, F-77447, Marne-la-Vall\'ee, France}
\email{colin.petitjean@univ-eiffel.fr}

\author[A. Proch\'azka]{Anton\'in Proch\'azka}
\address[A. Proch\'azka]{Laboratoire de Math\'ematiques de Besan\c con,
Universit\'e Bourgogne Franche-Comt\'e,
CNRS UMR-6623,
16, route de Gray,
25030 Besan\c con Cedex, France}
\email{antonin.prochazka@univ-fcomte.fr}

\begin{abstract}
We characterize compact metric spaces whose locally flat Lipschitz functions separate points uniformly as exactly those that are purely 1-unrectifiable, resolving a problem of Weaver. We subsequently use this geometric characterization to answer several questions in Lipschitz analysis. Notably, it follows that the Lipschitz-free space $\mathcal{F}(M)$ over a compact metric space $M$ is a dual space if and only if $M$ is purely 1-unrectifiable. Furthermore, we establish a compact determinacy principle for the Radon-Nikod\'ym property (RNP) and deduce that, for any complete metric space $M$, pure 1-unrectifiability is actually equivalent to some well-known Banach space properties of $\mathcal{F}(M)$ such as the RNP and the Schur property. A direct consequence is that any complete, purely 1-unrectifiable metric space isometrically embeds into a Banach space with the RNP. Finally, we provide a possible solution to a problem of Whitney by finding a rectifiability-based description of 1-critical compact metric spaces, and we use this description to prove the following: a bounded turning tree fails to be 1-critical if and only if each of its subarcs has $\sigma$-finite Hausdorff 1-measure.
\end{abstract}



\maketitle

\vspace{-0.1cm} 
\section{Introduction}

The motivation for the results of this paper comes from three corners of Lipschitz analysis: the Schur and Radon-Nikod\'ym properties (RNP) in Lipschitz-free Banach spaces, 
bi-Lipschitz embeddings into RNP spaces, and the geometry of Whitney arcs. As it turns out, the concept of a locally flat Lipschitz function is central to each one. Recall that a Lipschitz function $f: M \to \RR$ on a metric space $(M,d)$ is
\textit{locally flat} if
$$ \lim_{x,y \to p} \frac{|f(x)-f(y)|}{d(x,y)} = 0 $$
for every $p\in M$. When $M$ is compact, the limit condition can be equivalently replaced by $d(x,y)\rightarrow 0$, uniformly in $p$. 
In \cite{Weaver_1996}, Weaver isolated a fundamental property connecting locally flat Lipschitz functions on a metric space to its geometry. Following \cite[Defintion 4.10]{Weaver2}, we say that the locally flat Lipschitz functions \textit{separate points (of $M$) uniformly} if there exists $C \in [1,\infty)$ such that for any $x,y \in M$, there exists a $C$-Lipschitz locally flat function $f: M \to \RR$ with $f(x)-f(y) = d(x,y)$. The infimum of all such $C$ is called the \textit{separation constant}. Despite the importance of this property, the exact conditions on $M$ ensuring its satisfaction or failure remained unclear. Indeed, Weaver writes in \cite[p. 77-78]{Weaver1}, ``\textit{Whether [the locally flat Lipschitz functions] separate points uniformly depends only on the metric space $M$, so it is disappointing that the condition cannot be formulated in a way that directly involves the geometry of $M$.}" The first result of this article is a resolution of Weaver's problem. 
Recall that a metric space $M$ is \textit{purely $k$-unrectifiable} if, for every $A \subset \RR^k$ and Lipschitz map $f: A \to M$, the Hausdorff $k$-measure of $f(A)$ equals 0. By results in \cite{Kirchheim_1994}, $M$ is purely 1-unrectifiable if and only if it contains no \textit{curve fragments} - meaning bi-Lipschitz copies of compact, positive measure subsets of $\RR$ (see also Corollary \ref{cor:p1u=nocf}).

\begin{maintheorem} \label{ThmA}
Let $M$ be a compact metric space. Then the locally flat Lipschitz functions on $M$ separate points uniformly if and only if $M$ is purely 1-unrectifiable.
\end{maintheorem}

In his deep study of Lipschitz functions on purely $k$-unrectifiable metric spaces, Bate proved a closely related version of Theorem~\ref{ThmA} (\cite[Lemma 3.4]{Bate_2020}), and indeed Theorem \ref{ThmA} can be deduced in a straightforward manner from Bate's lemma. In Section \ref{section:lipspu}, we provide a self-contained proof.

In addition to its independent interest, Theorem \ref{ThmA} turns out to be the first step in a chain of implications that culminates in the characterization of prominent Banach space properties of Lipschitz-free spaces on the one hand, and a characterization of 1-critical compact metric spaces on the other hand. 
We will now describe these relations in turn, beginning with a review of Lipschitz-free spaces.

\medskip

\noindent \textbf{Lipschitz-free spaces.}
For a metric space $(M,d)$ equipped with a distinguished point $0 \in M$, the \textit{Lipschitz-free space} $\lipfree{M}$ is a Banach space that is built around $M$ in such a way that $M$ is isometric to a subset $\delta(M)$ of $\Free(M)$, and Lipschitz maps from $\delta(M)$ into any other Banach space $X$ uniquely extend to bounded linear operators from $\Free(M)$ into $X$ (see Section~\ref{Prelim:Free} for a more detailed definition).
In particular $\Free(M)$ is a canonical isometric predual of the space $\Lip_0(M)$ of Lipschitz functions on $M$ vanishing at $0$ endowed with the Lipschitz constant as a norm. Naturally, the study of Lipschitz-free spaces is at the intersection of functional analysis and metric geometry. Nevertheless, it is worth mentioning that Lipschitz-free spaces are studied in different research areas, for different reasons and under different names. For instance, thanks to the Kantorovich-Rubinstein duality theorem (e.g. \cite[Theorem~1.14]{Villani}), the norm on $\Free(M)$ can be interpreted as the cost of the optimal solution of a certain transportation problem (see \cite[Section~3.3]{Weaver2}, where $\lipfree{M}$ is called \textit{Arens-Eells space}).
They are also of significant interest for computer science where the names that are commonly used for this distance are \textit{earth mover distance} and \textit{transportation cost} (e.g. \cite{NaorSchecthman}). 

The first application of Theorem~\ref{ThmA} is a characterization of Lipschitz-free spaces over compact spaces which are isometrically dual Banach spaces. The study of duality of Lipschitz-free spaces $\lipfree{M}$ dates back to the 1960's. The first important results are due to Jenkins and Johnson \cite{JenkinsPhd,Johnson_1970}, who proved that for any compact $M$ endowed with a snowflake metric (also known as a H\"older metric), the relation
\begin{equation}
\label{eq:duality}
\lipfree{M} = \dual{\lip_0(M)}
\end{equation}
holds, where $\lip_0(M)$ is the ``little Lipschitz space''  consisting of all locally flat functions in $\Lip_0(M)$.
On the other hand, Lebesgue's fundamental theorem of calculus easily implies $\lip_0([0,1]) = \{0\}$, and so \eqref{eq:duality} cannot hold for $M = [0,1]$. More strongly, Godard proved in \cite{Godard_2010} that $\lipfree{A}$ is isomorphic to $L_1$ whenever $A$ is a positive measure subset of $\RR$, and therefore $\lipfree{M}$ can never be a dual space when $M$ is separable but not purely 1-unrectifiable. Some time after Johnson's result, Weaver identified in \cite{Weaver_1996} a powerful sufficient condition for \eqref{eq:duality}: it is enough that $\lip_0(M)$ separates points of $M$ uniformly (a property enjoyed by snowflaked metric spaces $M$).
Thus, Theorem \ref{ThmA} bridges the gap between the works of Weaver and Godard, and we arrive at the following theorem. 

\begin{maintheorem}
\label{ThmB}
Let $M$ be a compact metric space. Then the following are equivalent:
\begin{enumerate}[label={\upshape{(\roman*)}}]
\item $M$ is purely 1-unrectifiable,
\item $\lipfree{M}$ is a dual space,
\item $\lip_0(M)$ is an isometric predual of $\lipfree{M}$.
\end{enumerate}
\end{maintheorem}

Theorem \ref{ThmB} unifies a series of results giving sufficient conditions under which $\lipfree{M}$ is a dual space for compact $M$ \cite{APP_2019,Dalet_2015,Dalet_2015_2,Godard_2010,Weaver_1996} (see Section~\ref{section:duality} for details), and in particular it solves the long-standing question regarding whether $\lip_0(M)$ must be one of its preduals in that case. We also show that, while $\lip_0(M)$ is never a unique predual of $\lipfree{M}$ (unless $M$ is finite), it is the only predual that satisfies certain additional conditions (see Theorem \ref{th:lattice_predual}). Section \ref{section:duality} contains the proof of Theorem \ref{ThmB} (through Theorem \ref{th:seven_equivalences}) and further discussion on duality of Lipschitz-free spaces.

When combined with previously known implications, Theorem~\ref{ThmB} characterizes some well-known Banach space properties in Lipschitz-free spaces $\lipfree{M}$ for compact $M$. These properties include the Radon-Nikod\'ym property, the Krein-Milman property and the Schur property (see Theorem \ref{th:seven_equivalences} and the discussion thereafter). Recall that a Banach space $X$ has the \textit{Radon-Nikod\'ym property} (RNP) if every Lipschitz map $\RR \to X$ is differentiable Lebesgue-almost everywhere; it has the \textit{Schur property} if every weakly convergent sequence in $X$ is norm convergent; and it has the \textit{Krein-Milman property} if every closed bounded convex set in $X$ is the closed convex hull of its extreme points. 

Our second main result is the generalization of these equivalences to noncompact $M$. It was recently shown in \cite{ANPP_2020} that certain Banach space properties of Lipschitz-free spaces, including the Schur property, are compactly determined in the following sense: a Lipschitz-free space $\lipfree{M}$ has the mentioned property if and only if the subspace $\lipfree{K}$ has it for each compact $K\subset M$. This makes it possible to establish some results on Lipschitz-free spaces by reducing their proofs to the case where $M$ is compact. 
In Section \ref{section:rnp}, we adapt the methods from \cite{ANPP_2020} to the study of the Radon-Nikod\'ym property of Lipschitz-free spaces, based on its characterization in terms of martingales, and conclude that the RNP is also compactly determined (see Corollary \ref{cor:CR-RNP}). As a consequence, we obtain the following general result.

\begin{maintheorem}
\label{ThmC}
Let $M$ be a  metric space. Then the following are equivalent:
\begin{enumerate}[label={\upshape{(\roman*)}}]
\item The completion of $M$ is purely 1-unrectifiable,
\item $\lipfree{M}$ has the Radon-Nikod\'ym property,
\item $\lipfree{M}$ has the Krein-Milman property,
\item $\lipfree{M}$ has the Schur property,
\item $\lipfree{M}$ contains no isomorphic copy of $L_1$.
\end{enumerate}
\end{maintheorem}

Theorem \ref{ThmC} is a culmination of many previous works that provided sufficient conditions for a Lipschitz-free space to have the Schur property and suggested that it would be equivalent to the Radon-Nikod\'ym property \cite{ANPP_2020,Gartland_2020,HLP_2016,Kalton_2004,Petitjean_2017} (see Section~\ref{section:rnp} for more information and for the proof of Theorem~\ref{ThmC} through Theorem~\ref{thm:equivalences}).
To the best of our knowledge, Theorem~\ref{ThmC} is the first non-trivial characterization of isomorphic Banach space properties for $\Free(M)$ in terms of metric properties of $M$ (see e.g. \cite{APP_2019, Strong_Norm_Attaining, GLPRZ, Godard_2010, PRZ_2018} for characterizations of some isometric properties and \cite{Ambrosio} for a functional characterization of an isomorphic property). On top of that, it is a long-standing open problem in Banach Space theory whether the Radon-Nikod\'ym and Krein-Milman properties are equivalent in general (see e.g. \cite[p. 633]{FLP_handbook}), and Theorem \ref{ThmC} solves it for the particular case of Lipschitz-free spaces. The equivalence of several other properties follows easily; see Remark \ref{remark:more_equivalences} for a more detailed account. 
\bigskip

\noindent \textbf{Embeddings into RNP spaces.}
Our next major application is to the theory of bi-Lipschitz embeddings. Banach spaces with the Radon-Nikod\'ym property have gained popularity among metric space geometers because many well-known examples of metric spaces, such as the Heisenberg group or Laakso space, fail to bi-Lipschitz embed into any one of them (\cite[Theorem 1.6]{Cheeger-Kleiner}). The two main methods used to prove non-bi-Lipschitz embeddability of metric spaces into Banach spaces with the RNP are due to Cheeger and Kleiner (\cite[Theorem 1.6]{Cheeger-Kleiner}) and Ostrovksii (\cite[Theorem 1.3]{Ostrovskii}). The common feature of these methods is that the metric spaces under consideration must possess a large collection of curve fragments (see \cite{Bate_2014} and the last paragraph starting on page 2 of \cite{BateLi_2017}), and hence are far from being purely 1-unrectifiable. The following theorem is a partial converse to the theories of Cheeger-Kleiner and Ostrovskii. It is immediately implied by Theorem~\ref{ThmC} and the fact that every metric space isometrically embeds into its Lipschitz-free space.

\begin{maintheorem}
\label{ThmD}
A complete, purely 1-unrectifiable metric space isometrically embeds into a Banach space with the Radon-Nikod\'ym property.
\end{maintheorem} 

\noindent \textbf{Whitney sets.}
Our final application is to the theory of Whitney sets, or 1-critical sets. In \cite{Whitney_1935}, Whitney constructed a $C^1$ function $f: \RR^2 \to \RR$ such that $\nabla f$ vanishes on an arc $\gamma \subset \RR^2$, but $f$ is not constant on $\gamma$. Following the terminology from \cite{Norton_1989}, connected subsets $A$ of $\RR^n$ for which there exists a $C^1$ function $f: \RR^n \to \RR$ that is nonconstant on $A$ and $\nabla f \equiv 0$ on $A$ are known as \textit{1-critical sets}. These sets are very well-studied - for example, see \cite{Choquet_1944,CKZ_2008,PH_2003,Norton_1989,WX_2009,Whitney_1935}. By Whitney's extension theorem \cite{Whitney_1934}, a compact subset of $\RR^n$ is 1-critical if and only if it supports a nonconstant locally flat Lipschitz function, and thus 1-criticality can be defined as a purely metric notion in this way. Of course, if $M$ is rectifiably-connected, meaning every pair of points can be joined by a finite-length curve, then $M$ fails to be 1-critical. The obvious question is whether some sort of converse is true, and this problem was already posed in Whitney's original paper \cite{Whitney_1935} where he wrote ``\textit{it would be interesting to discover how far from rectifiable a closed set must be [to be 1-critical]}" (see also \cite[Q]{Norton_1989}). We provide a possible solution to Whitney's problem as an application of Theorem~\ref{ThmA}.

\begin{maintheorem} \label{ThmE}
A compact metric space fails to be 1-critical if and only if it is transfinitely almost-rectifiably-connected.
\end{maintheorem}

\noindent See Section \ref{section:Whitney} for an explanation of transfinite almost-rectifiable-connectedness. Using Theorem \ref{ThmE}, we are able to prove a quantitative, measure-theoretic characterization of the 1-criticality property in bounded turning trees (see Definition~\ref{def:bttree}), which includes the class of quasiarcs.

\begin{maintheorem} \label{ThmF}
Let $T$ be a 1-bounded turning tree, and for any $x,y \in T$, let $[x,y]$ denote the unique arc joining $x$ and $y$. For all $x,y \in T$,
$$\sup_{\substack{f \in \ball{\lip(T)}}} |f(y)-f(x)| = \inf\,\{\HH^1_\infty(A)\;:\; [x,y] \setminus A \text{ is } \HH^1\text{-}\sigma\text{-finite}\}$$
where $\ball{\lip(T)}$ is the set of locally flat 1-Lipschitz functions on $T$. In particular, a bounded turning tree fails to be 1-critical if and only if each of its subarcs is $\HH^1$-$\sigma$-finite.
\end{maintheorem}
\noindent See Section \ref{section:Whitney} for the proofs of Theorems \ref{ThmE} (as Corollary~\ref{cor:Whitney}) and \ref{ThmF} (as Theorem~\ref{thm:bttree}).
Each one of the Theorems \ref{ThmA} - \ref{ThmF} (except \ref{ThmD}) is new even for metric spaces that are subsets of a Euclidean space.
\medskip

\noindent \textbf{Notation.}
Throughout the paper, $M$ will stand for a metric space with metric $d$. We will assume without mention that $M$ is \textit{pointed}, i.e. we have selected a distinguished point $0\in M$.
We will use the notation
\begin{align*}
d(p,A) &= \inf\set{d(p,x):x\in A} \\
[A]_{r} &= \set{x \in M : d(x,A)\leq r} \\
B_r(p) &= \set{x \in M : d(x,p)\leq r} \\
\rad(A) &= \sup\set{d(x,0):x\in A} \\
\diam(A) &= \sup\set{d(x,y):x,y\in A}
\end{align*}
for $p\in M$, $A\subset M$ and $r\geq 0$. 
For convenience of the reader, let us recall the vector spaces
\begin{align*}
\Lip(M)   &= \set{f \in \RR^M : f \mbox{ Lipschitz}} \\
\lip(M)   &= \set{f \in \Lip(M)  : f \mbox{ locally flat}} \\
\Lip_0(M) &= \{f \in \Lip(M): f(0) = 0\} \\
\lip_0(M) &= \{f \in \lip(M): f(0) = 0\}.
\end{align*}
Whenever $V$ is one of the four spaces above, we have the unit ball
\begin{align*}
\ball{V} &= \{f \in V: f \text{ is } 1\text{-Lipschitz}\}.
\end{align*}
Note that the meaning of $\lip(M)$ and $\lip_0(M)$ for noncompact $M$ is not entirely consistent throughout the existing literature. For instance, some authors (e.g. \cite{Kalton_2004}) have used $\lip_0(M)$ to denote the Lipschitz functions $f:M \to \RR$ satisfying the uniform local flatness property
$$\lim_{d(x,y) \to 0} \frac{|f(x)-f(y)|}{d(x,y)} = 0.$$
When $M$ is not compact, this definition is more restrictive than the one we give, but this is of little consequence since our main results concerning $\lip(M)$ are stated only for compact spaces (except for a short departure into proper spaces in Section~\ref{section:duality} where we also require functions in $\lip(M)$ to be flat at infinity; see Definition \ref{def:flat_at_infinity}).

We use the lattice-theoretic notation $f \vee g$ and $f \wedge g$ for the pointwise maximum and minimum of the functions (or constants) $f$ and $g$, respectively. It holds that $f \vee g,f \wedge g \in \ball{\Lip(M)}$ whenever $f,g \in \ball{\Lip(M)}$, and similarly for $\ball{\lip(M)}$.

\subsection{Preliminaries on Lipschitz spaces and Lipschitz-free spaces}
\label{Prelim:Free}
For a metric space $M$ with a distinguished base point $0 \in M$, the \textit{Lipschitz-free space} or \textit{Arens-Eells space} $\lipfree{M}$ is a Banach space constructed around $M$ which is characterized by the following property: any Banach space-valued Lipschitz map $f : M \to X$ vanishing at $0$ can be extended in a unique way to a continuous linear map $\widehat{f} : \lipfree{M} \to X$ whose operator norm is equal to the best Lipschitz constant of $f$. This is often referred to as the ``universal extension property'' of Lipschitz-free spaces. 

There are several ways to construct the Lipschitz-free space over $M$ (see e.g. \cite[Chapter~8]{LipschitzBook}). We will focus on the following one: consider the Banach space $\Lip_0(M)$ equipped with the norm
$$
\lipnorm{f}=\sup\set{\frac{f(x)-f(y)}{d(x,y)} \;:\; x \neq y\in M}
$$
(note that $\lipnorm{\cdot}$ is not a norm on $\Lip(M)$, and it is for this reason that we work with $\Lip_0(M)$ instead). There are evaluation functionals $\delta(x)\in\dual{\Lip_0(M)}$ given by $\delta(x)\colon f\mapsto f(x)$ for $x\in M$, and $\lipfree{M}$ can be realized as the norm-closed linear span of $\set{\delta(x): x \in M}$ in the Banach space $\Lip_0(M)^\ast$. An easy consequence of the universal extension property (taking $X=\RR$) is that $\lipfree{M}$ is an isometric predual of $\Lip_0(M)$, and the corresponding weak$^\ast$ topology on $\ball{\Lip_0(M)}$ coincides with the topology of pointwise convergence. When $M$ is compact, it also agrees with the topology of uniform convergence. Another consequence is that, for any subset $M'\subset M$ containing $0$, $\lipfree{M'}$ may be canonically identified with the closed subspace of $\lipfree{M}$ generated by the evaluation functionals on points of $M'$. These facts will be used repeatedly in the sequel.

The set $\lip_0(M)$ is a norm-closed subspace of $\Lip_0(M)$, and generally it is not weak$^\ast$-closed. 
Another fact we will use implicitly is that $g \circ f \in \lip(M)$ whenever $f:M \to N$ is Lipschitz and $g \in \lip(N)$.

We refer the reader to \cite{Weaver2} for basic properties of Lipschitz-free spaces and to \cite{Godefroy_2015} for a survey on their applications to the nonlinear geometry of Banach spaces.

\subsection{Preliminaries on Hausdorff measure and Hausdorff convergence}
We review the basic properties of Hausdorff measure and content and Hausdorff metric and convergence. For more details, we refer the reader to \cite{Mattila} and \cite{Hyperspaces}.

\begin{definition}[Hausdorff Measure and Content]
For $M$ a separable metric space and $\delta \in (0,\infty]$, define
$$\mathcal{H}^1_\delta(M) := \inf\left\{\sum_{i=1}^\infty \diam(E_i) \;:\; M \subset \bigcup_{i=1}^\infty E_i ,\: \diam(E_i) < \delta\right\},$$
and
$$\mathcal{H}^1(M) := \lim_{\delta \to 0} \mathcal{H}^1_\delta(M).$$
$\mathcal{H}^1(M)$ is called the \textit{Hausdorff 1-measure} of $M$ and $\mathcal{H}^1_\infty(M)$ its \textit{Hausdorff 1-content}.
\end{definition}

\begin{remark}
By replacing each $E_i$ by $[E_i]_{\eps_i}$ with suitable $\eps_i>0$, we see that the definition of $\HH^1_\delta(M)$ is unchanged if we require that the interiors of the sets $E_i$ cover $M$. Hence if $M$ is compact, the definition of $\HH^1_\delta(M)$ is unchanged if we require the cover $\{E_i\}_i$ to be finite instead of countably infinite. We use this well-known fact in Lemma \ref{lem:usc} and in Section \ref{section:lipspu}.

\medskip
As the name suggests, Hausdorff 1-measure is a Borel measure on any metric space. While $\HH_\delta^1$ is in general not a measure for $\delta > 0$, it is $\sigma$-subadditive. Clearly, it always holds that $\HH^1(M) \geq \HH^1_\infty(M)$. The reverse inequality is not true in general, but it clearly does hold that $\HH^1_\infty(M) = 0 \Longleftrightarrow \HH^1(M) = 0$. Moreover, bi-Lipschitz mappings preserve the property of having null Hausdorff 1-measure. Let us also remark that, for $M=\RR$, $\HH^1$ agrees with the Lebesgue measure, which we shall denote by $\lambda$. We will use these facts without mention throughout the document.
\end{remark}

\begin{definition}
For $(\Omega,d)$ a compact metric space, the space of compact subsets of $\Omega$ is called the \textit{hyperspace} of $\Omega$. It is itself a compact metric space (see \cite[Theorems 3.1 and 3.5]{Hyperspaces}) with respect to the \textit{Hausdorff metric} defined by
$$d_H(K_1,K_2) := \inf\{\eps > 0 \;:\; K_1 \subset [K_2]_{\eps},\; K_2 \subset [K_1]_{\eps}\}.$$
\end{definition}

\begin{remark}
It is easy to verify that whenever $X$ is a compact metric space, $K_j,K \subset \Omega$ are compact with $K_j \to K$ with respect to the Hausdorff metric, and $f_j,f: \Omega \to X$ are continuous with $f_j \to f$ uniformly, then $f_j(K_j) \to f(K)$ with respect to the Hausdorff metric on the hyperspace of $X$.
\end{remark}

An important feature of Hausdorff content is its upper semi-continuity with respect to Hausdorff convergence, a feature not enjoyed by Hausdorff measure.

\begin{lemma}[Upper Semi-Continuity of Hausdorff Content]
\label{lem:usc}
Let $(\Omega,d)$ be a compact metric space. Suppose $K_j \to K$ with respect to the Hausdorff metric on the hyperspace of $\Omega$. Then
$$\mathcal{H}^1_\infty(K) \geq \limsup_{j \to \infty} \mathcal{H}^1_\infty(K_j).$$
\end{lemma}

\begin{proof}
Let $\eps > 0$ be arbitrary. Choose $E_1, \dots , E_n \subset \Omega$ such that $K \subset \bigcup_{i=1}^n E_i$ and $\mathcal{H}^1_\infty(K) + \eps \geq \sum_{i=1}^n \diam(E_i)$. Since $K_j \to K$, there is $J \in \NN$ so that $K_j \subset [K]_{\eps/n}$ for all $j \geq J$. Then $K_j \subset \bigcup_{i=1}^n [E_i]_{\eps/n}$, and hence
\begin{align*}
\limsup_{j \to \infty} \mathcal{H}^1_\infty(K_j) = \inf_{J \in \NN}\sup_{j \geq J}\mathcal{H}^1_\infty(K_j) &\leq \sum_{i=1}^n \diam([E_i]_{\eps/n}) \\
&\leq 2\eps + \sum_{i=1}^n \diam(E_i) \leq 3\eps + \mathcal{H}_\infty^1(K).
\end{align*}
Since $\eps > 0$ was arbitrary, the conclusion follows. 
\end{proof}

\subsection{Preliminaries on rectifiability and Kirchheim's theorems}
\label{Prelim:Kirchheim}

In \cite{Kirchheim_1994}, Kirchheim proved a variety of results concerning rectifiability in metric spaces and differentiation of Lipschitz maps that have since become standard tools in geometric measure theory. Before recalling his results, we introduce the main objects.

\begin{definition}
\label{def:unrectifiable}
Let $(M,d)$ be a metric space. We say that $M$ is \textit{$1$-rectifiable} if it equals the union of countably many Lipschitz images of subsets of $\RR$, up to a set of null $\HH^1$ measure. We say that $M$ is \textit{purely $1$-unrectifiable} if it contains no $1$-rectifiable subset of positive $\HH^1$ measure; equivalently, if $\HH^1(\gamma(A))=0$ for every $A\subset\RR$ and every Lipschitz $\gamma:A\to M$.
\end{definition}

\begin{definition}
Let $A\subset\RR$. The \textit{metric differential} of a Lipschitz map $\gamma: A \to M$ at a limit point $x \in A$ is defined to be the nonnegative number
$$MD(\gamma,x) := \lim_{\substack{y \to x \\ y \in A}} \dfrac{d(\gamma(x),\gamma(y))}{|x-y|},$$
whenever this limit exists. The set of all $x \in A$ for which $MD(\gamma,x)$ exists is denoted $\mathscr{MD}(\gamma)$, and the set of all $x \in \mathscr{MD}(\gamma)$ for which $MD(\gamma,x) > 0$ is denoted $\mathscr{MD}_r(\gamma)$. It follows from \cite[Theorem 2]{Kirchheim_1994} that $\mathscr{MD}(\gamma)$ has full $\lambda$ measure inside $A$.
\end{definition}

There are two fundamental theorems needed for our purposes. Fix a Lipschitz map $\gamma: A \to M$ with $A \subset \RR$ Borel (which implies $\mathscr{MD}(\gamma),\mathscr{MD}_r(\gamma)$ are Borel). 

\begin{theorem}[Bi-Lipschitz Decomposition; \cite{Kirchheim_1994}, Lemma 4]\label{thm:biLipdecomp}
For every $\eps > 0$, there exist a countable set of positive numbers $\{c_i\}_i$ and a countable collection of Borel subsets $\{E_i\}_i$ of $A$ such that
\begin{enumerate}[label={\upshape{(\roman*)}}]
    \item $\bigcup_i E_i = \mathscr{MD}_r(\gamma)$ and
    \item for every $x,y \in E_i$,
    \begin{equation}
    \label{eq:biLipdecomp}
    (1+\eps)^{-1}c_i|x-y| \leq d(\gamma(x),\gamma(y)) \leq (1+\eps)c_i|x-y| .
    \end{equation}
\end{enumerate}
\end{theorem}

\begin{theorem}[Area Formula; \cite{Kirchheim_1994}, Corollary 8]\label{thm:areaformula}
For each $y \in \gamma(A)$, let $|\gamma^{-1}(y)| \in [1,\infty]$ denote the cardinality of the preimage of $\{y\}$ under $\gamma$. Then
$$\int_A MD(\gamma,x)\, d\lambda(x) = \int_{\gamma(A)} |\gamma^{-1}(y)|\, d\HH^1(y).$$
Consequently, $MD(\gamma,x) = 0$ for $\lambda$-almost every $x \in A$ if and only if $\HH^1(\gamma(A)) = 0$.
\end{theorem}

\begin{remark}
The definition of the metric differential and all of Kirchheim's theorems are stated in \cite{Kirchheim_1994} for Lipschitz maps $\gamma$ whose domain is $\RR$ ($\RR^n$, in fact) and whose codomain is a Banach space. However, since every metric space $M$ isometrically embeds into $\ell_\infty(M)$ and every Lipschitz map $\gamma: A \to \ell_\infty(M)$ extends to a Lipschitz map on $\RR$, the results easily generalize to our formulation (see the last paragraph starting on page 113 of \cite{Kirchheim_1994}).
Indeed, to obtain these generalizations, it is enough to observe that the metric differential of $\gamma: A \to X$ and the metric differential of any Lipschitz extension $\tilde{\gamma}: \RR \to X$ agrees for $\lambda$-almost every $x \in A$ by \cite[Theorem 2]{Kirchheim_1994}.
\end{remark}

The next lemma and its corollary are fundamental to our paper and can be easily derived from Kirchheim's theorems. While these are very well known and variants are commonly used in the field (e.g., \cite[Theorem 11.12]{DS_1997} and \cite[Lemma 7.2]{Bate_2020}), we discuss the derivation for the convenience of the reader.

\begin{lemma}\label{lem:Kirchheim}
For every metric space $(M,d)$, Lipschitz map $\gamma: A \to M$ with $A \subset \RR$ Borel and $\HH^1(\gamma(A)) > 0$, and $\eps > 0$, there exist $c > 0$ and a Borel subset $E \subset A$ with $\lambda(E) > 0$ such that $\gamma \restrict_E$ is a $(1+\eps)$-bi-Lipschitz embedding into the metric space $(M, c \cdot d)$. 
\end{lemma}

\noindent Here is how one may derive the lemma. Let $\gamma: A \to M$ be Lipschitz with $\HH^1(\gamma(A)) > 0$, and let $\eps>0$. Then $\lambda(\mathscr{MD}_r(\gamma)) > 0$ by Theorem \ref{thm:areaformula}, and so by Theorem \ref{thm:biLipdecomp} there exist $E_i \subset A$ Borel and $c_i > 0$ such that $\lambda(E_i)>0$ and the inequalities \eqref{eq:biLipdecomp} hold for all $x,y \in E_i$. Then $E := E_i$ and $c := c_i^{-1}$ satisfy the required properties.

Recall that a \textit{curve fragment} is the image of a bi-Lipschitz embedding $\gamma: K \to M$, where $K \subset \RR$ is compact with $\lambda(K) > 0$. In an abuse of terminology, we may also refer to the embedding itself as a curve fragment. The next corollary is used frequently and without mention throughout the article. It
follows easily from Lemma \ref{lem:Kirchheim} and inner regularity of $\lambda$.

\begin{corollary}\label{cor:p1u=nocf}
A metric space is purely 1-unrectifiable if and only if it contains no curve fragment.
\end{corollary}


\section{Proof of Theorem \ref{ThmA}}
\label{section:lipspu}

We recall Theorem \ref{ThmA} and then discuss its proof.

\begin{reptheorem}{ThmA}
Let $M$ be a compact metric space. Then the locally flat Lipschitz functions on $M$ separate points uniformly if and only if $M$ is purely 1-unrectifiable.
\end{reptheorem}

Let us first remark that the proof of the ``only if" implication is easy and well-known. Indeed, if $K$ bi-Lipschitz embeds into $M$ and $\lip(M)$ separates the points of $M$ uniformly, then $\lip(K)$ separates the points of $K$ uniformly. Using Lebesgue's density theorem and fundamental theorem of calculus, it is not difficult to check that, whenever $K \subset \RR$ is compact with positive measure, $\lip(K)$ does not separate the points of $K$ uniformly (see \cite[Example~4.13(b)]{Weaver2} for details). Consequently, $\lip(M)$ cannot separate the points of $M$ uniformly if $M$ contains a curve fragment $\gamma(K)$. In fact, this argument works when $M$ is any metric space; compactness isn't necessary. The real content of Theorem \ref{ThmA} is the ``if" direction, which is implied by the following seemingly stronger statement. 

\begin{theorem}
\label{thm:spu}
If $M$ is compact and purely 1-unrectifiable then, for all $p\in M$ and $\delta > 0$, there exists $g \in \ball{\lip(M)}$ such that $g(x) - g(p) \geq d(p,x) - \delta$ for every $x \in M$.
\end{theorem}

We use the remainder of this section to prove Theorem \ref{thm:spu}. The proof occurs in the last subsection, following a host of supporting lemmas. First, we recall a useful method for constructing Lipschitz functions on a metric space $M$ with prescribed local behavior. We describe the method here and summarize the conclusion in Proposition \ref{prop:locallyflat}. The method requires $M$ to be isometrically embedded in a convex subset $\Omega$ of a Banach space (actually, a geodesic metric space would suffice).
However, there is no loss of generality in making this assumption, because every metric space $N$ isometrically embeds into the Lipschitz-free space $\lipfree{N}$. Additionally, when $M$ is compact, there is no loss of generality in assuming that $\Omega$ is also compact, because closed convex hulls of compact subsets of Banach spaces are compact. Before stating the construction, we briefly review the definition of path integrals and length measure for the purpose of setting notation.

\begin{definition}[Length Measure and Path Integrals]
Let $(\Omega,d)$ be a metric space. When $\gamma: [a,b] \to \Omega$ is a Lipschitz curve, we get a \textit{total variation measure} $TV_\gamma$ on $[a,b]$ defined on intervals by
$$
TV_{\gamma}([s,t]) := \sup\set{\sum_{i=1}^n d(\gamma(t_{i-1}),\gamma(t_i)) \;:\; s = t_0 < t_1 < \dots < t_n = t} .
$$
Since $\gamma$ is Lipschitz, $TV_\gamma \leq L\cdot\lambda$ for some $L < \infty$. Pushing forward the measure $TV_\gamma$ under $\gamma$ gives a finite, positive Borel measure $\mu_\gamma$ on $\Omega$, called the \textit{length measure} of $\gamma$. The length measure is invariant with respect to Lipschitz reparametrizations. When $f: \Omega \to \RR$ is bounded Borel, we get a \textit{path integral} defined by
$$\int_\gamma f\, ds := \int_\Omega f\, d\mu_\gamma = \int_a^b (f \circ \gamma)\, dTV_\gamma.$$
The \textit{length} of $\gamma$ is $|\gamma| := \int_\gamma 1\, ds = \mu_\gamma(\Omega) = TV_\gamma([a,b])$.
\end{definition}

We are now ready to state our main method of constructing locally flat Lipschitz functions. The construction is a simplified variation of the one used by Bate in \cite[Section 3]{Bate_2020}.

\begin{definition} \label{def:phidef}
Given a convex subset $\Omega \subset X$ of a Banach space $X$, $p \in \Omega$, and a bounded Borel function $f: \Omega \to [0,\infty)$, we define a function ${\phi_{f,p}: \Omega \to [0,\infty)}$ by
$$\phi_{f,p}(x) := \inf_\gamma \int_\gamma f\, ds,$$
for $x\in\Omega$, where the infimum is over all $a \leq b \in \RR$ and Lipschitz curves $\gamma: [a,b] \to \Omega$ with $\gamma(a) = p$ and $\gamma(b) = x$.
\end{definition}

The function $\phi_{f,p}$ should be thought as an ``antiderivative" of $f$ of sorts. Note that $\phi_{f,p}(p)$ vanishes at $p$. Let us quickly check that $\phi_{f,p}$ is Lipschitz. Let $x,y \in \Omega$. Without loss of generality, we may assume $\phi_{f,p}(x) \leq \phi_{f,p}(y)$. Let $\eps > 0$, and choose $\gamma: [a,b] \to \Omega$ with $\gamma(a) = p$, $\gamma(b) = x$, and $\int_\gamma f\, ds < \phi_{f,p}(x) + \eps$. Let
$$[x,y] := \{(1-t)x+ty \in X \;:\; t \in [0,1]\}$$
denote the line segment in $X$ connecting $x$ and $y$. Since $x,y \in \Omega$ and $\Omega$ is convex, $[x,y] \subset \Omega$. Let $\gamma_{[x,y]}: [b,b+\|x-y\|] \to [x,y]$ denote the unit speed parametrization starting at $x$ and ending at $y$. Then we create a new Lipschitz curve ${\tilde{\gamma}: [a,b+\|x-y\|] \to \Omega}$ by concatenating $\gamma$ with $\gamma_{[x,y]}$. Specifically,
$$
\tilde{\gamma}(t) := \begin{cases} \gamma(t) &\text{, if $t \in [a,b]$} \\ \gamma_{[x,y]}(t) &\text{, if $t \in [b,b+\|x-y\|]$} \end{cases} .
$$
It is easy to see that $\tilde{\gamma}$ is Lipschitz, $\tilde{\gamma}(a) = p$, and $\tilde{\gamma}(b+\|x-y\|) = y$. Thus, $\phi_{f,p}(y) \leq \int_{\tilde{\gamma}} f\, ds$. Then we have
\begin{align*}
|\phi_{f,p}(y)-\phi_{f,p}(x)| &= \phi_{f,p}(y)-\phi_{f,p}(x) < \int_{\tilde{\gamma}} f\, ds - \int_\gamma f\, ds + \eps\\
&= \int_{\gamma_{[x,y]}} f\, ds + \eps \leq \int_{\gamma_{[x,y]}} \|f\|_{L_\infty([x,y])}\, ds +\eps \\
&= \|f\|_{L_\infty([x,y])} |{\gamma_{[x,y]}}| + \eps = \|f\|_{L_\infty([x,y])} \|x-y\| + \eps.
\end{align*}
Since $\eps > 0$ was arbitrary, this shows
$$ |\phi_{f,p}(y)-\phi_{f,p}(x)| \leq \|f\|_{L_\infty([x,y])} \|x-y\| .$$
This inequality easily implies the following proposition.

\begin{proposition} \label{prop:locallyflat} 
Let $\Omega$ be a convex subset of a Banach space. For all $M \subset \Omega$, $p \in \Omega$, and $f: \Omega \to [0,1]$ Borel with $\lim\limits_{r \to 0} \sup\limits_{x \in [M]_r} f(x) = 0$, we have $\phi_{f,p}(p) = 0$ and $\phi_{f,p}\restrict_M \in \ball{\lip(M)}$.
\end{proposition}

\subsection{Constructing the Separating Locally Flat Lipschitz Function}

The next lemma is the linchpin of this section. We crucially use the compactness of $\Omega$ in its proof. It is also the only time we directly appeal to the pure 1-unrectifiability of $M$.

\begin{lemma}[Neighborhood Inducing Small Hausdorff Content]
\label{lem:V}
For every compact metric space $\Omega$, purely 1-unrectifiable closed subset $M \subset \Omega$, $L < \infty$, and $\eps > 0$, there exists a compact neighborhood $V$ of $M$ in $\Omega$ such that \linebreak ${\mathcal{H}_\infty^1(\im(\gamma) \cap V) < \eps}$ for every Lipschitz curve $\gamma: [a,b] \to \Omega$ with $|\gamma| \leq L$.
\end{lemma}

\begin{proof}
Let $\Omega$ and $M$ be as above, and suppose the lemma is false. Then we can find $L < \infty$, $\eps > 0$, and Lipschitz curves $\gamma_n: [a_n,b_n] \to \Omega$ with $|\gamma_n| \leq L$ and $\mathcal{H}_\infty^1(\im(\gamma_n) \cap [M]_{1/n}) \geq \eps$ for every $n$. By parametrizing by arclength on $[0,|\gamma_n|]$ and constant on $[|\gamma_n|,L]$, we may assume each $\gamma_n: [0,L] \to \Omega$ is 1-Lipschitz.
Set $K_n := \gamma_n^{-1}([M]_{1/n}) \subset [0,L]$, so that $\gamma_n(K_n) = \im(\gamma_n) \cap [M]_{1/n}$. 
By the Arzel\`a-Ascoli theorem, we may assume $\gamma_n$ converges to some 1-Lipschitz $\gamma: [0,L] \to \Omega$ uniformly. By compactness of the hyperspace of $[0,L]$, we may also assume that $K_n \to K$ with respect to the Hausdorff metric for some compact $K \subset [0,L]$. It follows that $\gamma_n(K_n) \to \gamma(K)$ with respect to the Hausdorff metric on the hyperspace of $\Omega$. Lemma \ref{lem:usc} implies $\mathcal{H}_\infty^1(\gamma(K)) \geq \eps$. It also holds that $\gamma(K) \subset M$. To see this, let $\eta > 0$ be arbitrary. Choose $n$ large enough so that $d_H(\gamma_n(K_n),\gamma(K)) < \eta$ and $1/n < \eta$. Then $\gamma(K) \subset [\gamma_n(K_n)]_\eta$ and $\gamma_n(K_n) \subset [M]_\eta$, hence $\gamma(K) \subset [M]_{2\eta}$. Since $M$ is closed and $\eta > 0$ was arbitrary, it follows that $\gamma(K) \subset M$. Thus, $\gamma: [0,L] \to \Omega$ is Lipschitz and
$$\mathcal{H}^1(\im(\gamma) \cap M) \geq \mathcal{H}^1(\gamma(K)) \geq \mathcal{H}_\infty^1(\gamma(K)) \geq \eps > 0,$$
contradicting pure 1-unrectifiability of $M$.
\end{proof}

Throughout the remainder of this subsection, let $\Omega$ be a compact, convex subset of a Banach space, with induced metric denoted $d$. Let $M$ be a closed, purely 1-unrectifiable subset of $\Omega$, and fix $\delta>0$.
Let $M \subset \dots \subset V_2 \subset V_1 \subset V_0 = \Omega$ be a decreasing sequence of compact neighborhoods of $M$ such that, for all $n \geq 0$ and Lipschitz curves $\gamma: [a,b] \to \Omega$,
\begin{equation} \label{eq:Vdef}
|\gamma| \leq n \quad\Longrightarrow\quad \HH_\infty^1(\im(\gamma) \cap V_n) < \delta 2^{-n}.
\end{equation}
Such a sequence exists by Lemma \ref{lem:V}.

In the proof of Theorem \ref{thm:spu}, we will use these neighborhoods to construct a bounded Borel function $f: \Omega \to [0,1]$ that is constant on $V_{n-1} \setminus V_n$, and the corresponding Lipschitz function $\phi_{f,p}$ will be used to fulfill Theorem \ref{thm:spu}. To obtain the necessary estimates on $\phi_{f,p}$, we need three lemmas about modifying Lipschitz curves. Each lemma builds off the previous one.

In what follows, if $r>0$ and $\mu_1$, $\mu_2$ are two Borel measures on $\Omega$, then by $\mu_1 \leq  \mu_2 +r$, we mean $\mu_1(A) \leq \mu_2(A) + r$ for every Borel set $A$. Notice that this is equivalent to $\int_\Omega f\,d\mu_{1} \leq \int_\Omega f\,d\mu_{2} + r\|f\|_\infty$ for every
Borel function $f : \Omega \to [0,+\infty)$. Indeed, one implication is trivial by setting $f = 1_A$. For the other, let $\Omega = P \sqcup N$ be a Hahn decomposition with respect to the signed measure $\mu_1-\mu_2$, then
\begin{align*}
\int_\Omega f\,d\mu_1 = \int_\Omega f\,d\mu_2 + \int_\Omega f\,d(\mu_1-\mu_2) &\leq \int_\Omega f\,d\mu_2 + (\mu_1-\mu_2)(P)\|f\|_\infty \\
&\leq \int_\Omega f\,d\mu_2 + r\|f\|_\infty \,.
\end{align*}

\begin{lemma}[Curve Modification in Set of Small Diameter]
\label{lem:ball} 
For any Lipschitz curve $\gamma: [a,b] \to \Omega$, $r \geq 0$, and compact, convex subset $E \subset \Omega$ of diameter $r$, there exists another Lipschitz curve $\tilde{\gamma}: [a,b] \to \Omega$ such that
\begin{itemize}
\item $\tilde{\gamma}$ has the same endpoint values as $\gamma$,
\item $\mu_{\tilde{\gamma}}(E) \leq r$, and
\item for every Borel $A \subset \Omega$, $\mu_{\tilde{\gamma}}(A \setminus E) \leq \mu_\gamma(A \setminus E)$.
\end{itemize}
Consequently, $\mu_{\tilde{\gamma}} \leq \mu_\gamma + r$.
\end{lemma}

\begin{proof}
Let $\gamma,r,E$ be as above. If $\im(\gamma) \cap E = \emptyset$, we simply choose $\tilde{\gamma} = \gamma$ and the proof is finished. So assume $\im(\gamma) \cap E \neq \emptyset$. Let
\begin{align*}
s_0 &:= \min\{t \in [a,b]: \gamma(t) \in E\} \\
s_1 &:= \max\{t \in [a,b]: \gamma(t) \in E\}
\end{align*}
which exist by compactness of $E$ and continuity of $\gamma$. Define $\tilde{\gamma}$ on $[a,s_0] \cup [s_1,b]$ to agree with $\gamma$, and on $[s_0,s_1]$ to be the constant speed parametrization of the line segment between $\gamma(s_0)$ and $\gamma(s_1)$. This line segment belongs to $E$ by convexity, and $\tilde{\gamma}$ satisfies the desired properties. 
\end{proof}

\begin{lemma}[Curve Modification in Set of Small Hausdorff Content]
\label{lem:smallHcontent}
For any Lipschitz curve $\gamma: [a,b] \to \Omega$, $\eps > 0$, and compact $K \subset \Omega$ with $\mathcal{H}^1_\infty(\im(\gamma) \cap K) < \eps$, there exists another Lipschitz curve $\tilde{\gamma}: [a,b] \to \Omega$ such that
\begin{itemize}
\item $\tilde{\gamma}$ has the same endpoint values as $\gamma$,
\item $\mu_{\tilde{\gamma}}(K) < \eps$, and
\item $\mu_{\tilde{\gamma}} < \mu_\gamma + \eps$.
\end{itemize}
\end{lemma}

\begin{proof}
Let $\gamma,\eps,K$ be as above. Since $\mathcal{H}^1_\infty(\im(\gamma) \cap K) < \eps$, there exist finitely many subsets $E_1, \dots E_n \subset \Omega$ of diameters $r_1, \dots r_n \geq 0$ such that $\im(\gamma) \cap K \subset \bigcup_{i=1}^n E_i$ and $\sum_{i=1}^n r_i < \eps$. Since the diameter of $E_i$ is the same as that of its closed convex hull, we may assume each $E_i$ is compact and convex. Apply Lemma \ref{lem:ball} with $\gamma = \gamma$, $r = r_1$, and $E = E_1$ to obtain a Lipschitz curve $\gamma_1: [a,b] \to \Omega$ satisfying
\begin{itemize}
\item $\gamma_1$ has the same endpoint values as $\gamma$,
\item $\mu_{\gamma_1}(E_1) \leq r_1$, and
\item $\mu_{\gamma_1} \leq \mu_\gamma + r_1$.
\end{itemize}
Then we apply Lemma \ref{lem:ball} again with $\gamma = \gamma_1$, $r = r_2$, and $E = E_2$ to obtain a Lipschitz curve $\gamma_2: [a,b] \to \Omega$ satisfying
\begin{itemize}
\item $\gamma_2$ has the same endpoint values as $\gamma_1$, which are the same as $\gamma$'s,
\item $\mu_{\gamma_2}(E_1 \cup E_2) = \mu_{\gamma_2}(E_1 \setminus E_2) + \mu_{\gamma_2}(E_2)\leq \mu_{\gamma_1}(E_1 \setminus E_2) + r_2 \leq r_1+r_2$, and
\item $\mu_{\gamma_2} \leq \mu_{\gamma_1} + r_2 \leq \mu_\gamma + r_1 + r_2$.
\end{itemize}
Repeating this up to $n$ times produces a Lipschitz curve $\gamma_n: [a,b] \to \Omega$ satisfying
\begin{itemize}
\item $\gamma_n$ has the same endpoint values as $\gamma$,
\item $\mu_{\gamma_n}\left(\bigcup_{i=1}^n E_i\right) \leq \sum_{i=1}^n r_i$, and
\item $\mu_{\gamma_n} \leq \mu_\gamma + \sum_{i=1}^n r_i$.
\end{itemize}
Since $\im(\gamma_n)\cap K \subset \bigcup_{i=1}^n E_i$ and $\sum_{i=1}^n r_i < \eps$, $\tilde{\gamma} = \gamma_n$ satisfies the desired properties.
\end{proof}

\begin{lemma}[Curve Modification in $V_n$]
\label{lem:neighborhood}
For any Lipschitz curve $\gamma: [a,b] \to \Omega$ with $\gamma(a) \neq \gamma(b)$, there exist another Lipschitz curve $\tilde{\gamma}: [a,b] \to \Omega$ and $m\in\NN$ such that
\begin{itemize}
\item $\tilde{\gamma}$ has the same endpoint values as $\gamma$,
\item $m-1 < |\tilde{\gamma}|$,
\item $\mu_{\tilde{\gamma}}(V_m) < \delta$, and
\item $\mu_{\tilde{\gamma}} < \mu_\gamma + \delta$.
\end{itemize}
Consequently, for every Borel function $f:\Omega \to [0,1]$ we have 
$$\int_\gamma f\,ds\geq \int_{\tilde{\gamma}}f\,ds- \delta \geq |\tilde{\gamma}|\inf_{\Omega\setminus V_m}f - 2\delta.$$
\end{lemma}

\begin{proof}
Let $\gamma$ be as above. Let $A$ be the set of all $n\in\NN$ such that there exists a Lipschitz curve $\tilde{\gamma}: [a,b] \to \Omega$ satisfying
\begin{itemize}
\item $\tilde{\gamma}$ has the same endpoint values as $\gamma$,
\item $\mu_{\tilde{\gamma}}(V_n) < \delta$, and
\item $\mu_{\tilde{\gamma}} < \mu_\gamma + \delta \sum_{i=n}^\infty 2^{-i}$.
\end{itemize}
Let $n := \lceil |\gamma| \rceil$ be the least integer greater than or equal to $|\gamma|$. Note that $n \geq 1$ since $\gamma(a) \neq \gamma(b)$. The definition \eqref{eq:Vdef} of $V_n$ implies $\mathcal{H}_\infty^1(\im(\gamma) \cap V_n) < \delta2^{-n}$. Then we apply Lemma \ref{lem:smallHcontent} with $\gamma = \gamma$, $\eps = \delta2^{-n}$, and $K = V_n$ to obtain a Lipschitz curve $[a,b] \to \Omega$ that witnesses $n \in A$. Thus $A \neq \emptyset$, and $m := \min(A)$ exists. Let $\tilde{\gamma}$ be a Lipschitz curve witnessing $m \in A$. We will show it must hold that $m-1 < |\tilde{\gamma}|$, which will finish the proof.

Since $\tilde{\gamma}$ has the same endpoint values as $\gamma$ and $\gamma(a) \neq \gamma(b)$, it holds that $|\tilde{\gamma}| > 0$. Hence, we are done if $m-1 = 0$.
In the remaining case, when $m-1\geq 1$, we assume towards a contradiction that $|\tilde{\gamma}| \leq m-1$. Then by definition of $V_{m-1}$, $\mathcal{H}_\infty^1(\im(\tilde{\gamma}) \cap V_{m-1}) < \delta 2^{-m+1}$. Apply Lemma \ref{lem:smallHcontent} with $\gamma = \tilde{\gamma}$, $\eps = \delta 2^{-m+1}$, and $K = V_{m-1}$ to obtain a Lipschitz curve $\gamma_0: [a,b] \to \Omega$ satisfying
\begin{itemize}
\item $\gamma_0$ has the same endpoint values as $\tilde{\gamma}$, which are the same as $\gamma$'s,
\item $\mu_{\gamma_0}(V_{m-1}) < \delta$, and
\item $\mu_{\gamma_0} < \mu_{\tilde{\gamma}} + \delta2^{-m+1} <  \mu_\gamma + \delta \sum_{i=m}^\infty 2^{-i} + \delta 2^{-m+1} = \mu_\gamma + \delta \sum_{i=m-1}^\infty 2^{-i}$.
\end{itemize}
Thus, $\gamma_0$ witnesses $m-1 \in A$ (together with our prior assumption that $m-1 \geq 1$). This contradicts $m = \min(A)$.
 
To prove the last part of the lemma, notice that $$\int_{\tilde{\gamma}} f\, ds = \int_\Omega f\, d\mu_{\tilde{\gamma}} \leq \int_\Omega f\, d\mu_{\gamma} + \delta\|f\|_\infty \leq \int_\gamma f\, ds + \delta.$$
Now, since $\mu_{\tilde{\gamma}}(V_m) < \delta$, we have
$$\int_{\tilde{\gamma}} \mathsf{1}_{\Omega \setminus V_m}\, ds \geq |\tilde{\gamma}| - \delta.$$
Thus, by positivity of $f$, we obtain
\begin{equation*}
\int_{\tilde{\gamma}} f\, ds \geq (|\tilde{\gamma}| - \delta) \inf_{\Omega\setminus V_m}f \geq |\tilde{\gamma}|\inf_{\Omega\setminus V_m}f - \delta.
\qedhere
\end{equation*}
 \end{proof}

We are now prepared to describe the construction of the appropriate bounded Borel function $f: \Omega \to [0,1]$ and use Lemma \ref{lem:V} to get necessary estimates on $\inf_\gamma \int_\gamma f\, ds$.

\subsection{Proof of Theorem \ref{thm:spu}}
\begin{proof}[Proof of Theorem \ref{thm:spu}.]
Let $p \in M$ and $\delta > 0$. Set $c_0 = 1$ and
$$c_n := 1 \wedge \frac{\diam(\Omega)}{n}$$
for $n \geq 1$. 
Define the Borel function $f: \Omega \to [0,1]$ by
$$f:=\begin{cases} c_n & \text{on $V_n \setminus V_{n+1}$} \\ 0 & \text{on $\bigcap_{n=0}^\infty V_n$.} \end{cases}$$
Then since $c_n \searrow 0$, $\lim\limits_{r \to 0} \sup\limits_{x \in [M]_r} f(x) = 0$. Hence, Proposition \ref{prop:locallyflat} implies $g := \phi_{f,p}\restrict_{M} \in \ball{\lip(M)}$.

Let $x \in \Omega$ and $\gamma: [a,b] \to \Omega$ be a Lipschitz curve such that $\gamma(a) = p$, $\gamma(b) = x$.
Then by Lemma \ref{lem:neighborhood}, there is a Lipschitz curve $\tilde{\gamma}: [a,b] \to \Omega$ and $n \geq 0$ (indeed, take $n=m-1)$ such that $\tilde{\gamma}(a) = p$, $\tilde{\gamma}(b) = x$, $n \leq |\tilde{\gamma}|$, and
\begin{equation} \label{eq:a}
\int_\gamma f\, ds  \geq c_n|\tilde{\gamma}| - 2\delta.
\end{equation}
Since $\tilde{\gamma}$ connects $p$ to $x$, it must hold that $d(p,x) \leq |\tilde{\gamma}|$, and thus
$$c_n|\tilde{\gamma}| \geq c_nd(p,x) \vee c_nn.$$
If $\diam(\Omega) \geq n$, then $c_nd(p,x) = d(p,x)$, and if $\diam(\Omega) \leq n$, then $c_nn = \diam(\Omega) \geq d(p,x)$. So, in all cases,
$$c_n|\tilde{\gamma}| \geq d(p,x).$$
This inequality, together with \eqref{eq:a}, the definition of $g$ and the fact that $\gamma$ was arbitrary  yield 
\[
g(x)-g(p)\geq d(p,x)-2\delta,
\]
which we wanted to prove.
\end{proof}


\section{Duality of Lipschitz-free spaces}
\label{section:duality}

We will now use the results from Section \ref{section:lipspu} to obtain a characterization of several Banach space properties in Lipschitz-free spaces over compact spaces. 
One of the goals is to characterize the compact metric spaces $M$ such that the Lipschitz-free space $\lipfree{M}$ is a dual. Not all compacts satisfy this; for instance, $\lipfree{[0,1]}=L_1$ is not a dual space. For all previously known examples where $\lipfree{M}$ is a dual, we actually have
\begin{equation}
\label{eq:duality2}
\lipfree{M} = \dual{\lip_0(M)}.
\end{equation}
Let us remark that $\lip_0(M)$ is usually not a unique predual: for instance, consider the well-known cases where $\lipfree{M}=\ell_1$ (see e.g. \cite[Example 3.10]{Weaver2}), which has a plethora of non-isomorphic preduals.

The study of duality of Lipschitz-free spaces $\lipfree{M}$ can be traced back to the 1960's. After the pioneering works of Cieselski and de Leeuw who studied particular cases (snowflakes of the unit interval and the unit circle, respectively \cite{Cieselski_1960,deLeeuw_1961}), the first important results are due to Jenkins and Johnson \cite{JenkinsPhd,Johnson_1970}, who proved that \eqref{eq:duality2} holds for any compact $M$ when endowed with a snowflake metric.
A powerful sufficient condition for \eqref{eq:duality2} was identified much later by Weaver in \cite{Weaver_1996}, based on observations in \cite{BCD_1987,Hanin_1992}: it is enough that $\lip_0(M)$ separates points of $M$ uniformly.
It is then immediate from Goldstine's theorem that the separation constant is $1$.
Using this condition, Dalet proved \eqref{eq:duality2} for compacts $M$ that are either countable \cite{Dalet_2015} or ultrametric \cite{Dalet_2015_2}. Weaver later gave a much simplified proof of Dalet's result for the countable case \cite[Theorem 4.11]{Weaver2}. 
In turn, three of the named authors used in \cite[Theorem~4.3]{APP_2019} an argument similar to Weaver's to show that \eqref{eq:duality2} holds whenever $M$ is a compact $0$-hyperbolic metric space whose length measure is 0. The same proof can in fact be used, with only slight modifications, to show that \eqref{eq:duality2} is true for any compact $M$ with $\HH^1(M) = 0$.

On the other hand, we may also identify necessary conditions for duality. By Lebesgue's fundamental theorem of calculus, there are no nontrivial locally flat functions on $[0,1]$, and therefore $\lip_0(M)$ is trivial whenever $M$ is geodesic. Godard provided a stronger necessary condition in \cite{Godard_2010} when he proved that $\lipfree{M}$ is isomorphic to $L_1$ for any $M\subset\RR$ with $\lambda(M)>0$. 
Since bi-Lipschitz equivalent metric spaces have isomorphic Lipschitz-free spaces, it follows that $\lipfree{M}$ cannot be a dual space if $M$ contains a curve fragment - that is, $M$ has to be purely 1-unrectifiable for $\lipfree{M}$ to be a dual. More generally, this fact holds for all separable $M$, not just compacts.

Returning to the compact setting, Theorem \ref{thm:spu} now allows us to fill the gap immediately as it shows that, when $M$ is compact, pure 1-unrectifiability of $M$ is also a sufficient condition for $\lip_0(M)$ to separate points of $M$ uniformly, and thus for \eqref{eq:duality2} to hold. In fact, combining it with previously known implications, we deduce that this characterizes some well-known Banach space properties in Lipschitz-free spaces, like the Radon-Nikod\'ym property or the Schur property.

It happens frequently in Lipschitz-free space theory that results for compact metric spaces can be extended to \textit{proper} metric spaces, i.e. such that every closed ball is compact. This is also the case with the characterization of duality, but it requires a slight modification of the definition of the little Lipschitz space.
For the remainder of this section, we will use the following notation:
\begin{definition}
\label{def:flat_at_infinity}
If $M$ is a proper metric space, $\lip_0(M)$ (resp. $\lip(M)$) will denote the space of all functions in $\Lip_0(M)$ (resp. $\Lip(M)$) that are locally flat and moreover \textit{flat at infinity}, i.e. such that
$$\lim_{r\to\infty}\lipnorm{f\restrict_{M\setminus B_r(0)}}=0 .$$
Note that flatness at infinity is automatically satisfied when $M$ is compact.
\end{definition}

\begin{theorem}
\label{th:seven_equivalences}
Let $M$ be a proper metric space. Then the following are equivalent:
\begin{enumerate}[label={\upshape{(\roman*)}}]
\item \label{e:p1u} $M$ is purely 1-unrectifiable,
\item \label{e:lip_spu} $\lip_0(M)$ separates points of $M$ uniformly,
\item \label{e:lip_predual} $\dual{\lip_0(M)}=\lipfree{M}$,
\item \label{e:dual} $\lipfree{M}$ is a dual space,
\item \label{e:rnp} $\lipfree{M}$ has the Radon-Nikod\'ym property,
\item \label{e:schur} $\lipfree{M}$ has the Schur property,
\item \label{e:no_L1} $\lipfree{M}$ contains no isomorphic copy of $L_1$.
\end{enumerate}
\end{theorem}

\begin{proof}
\ref{e:p1u}$\Rightarrow$\ref{e:lip_spu}: The compact case is given by Theorem \ref{thm:spu}. Let us now extend it to the proper case.
Fix two points $p,q\in M$ and $\delta>0$. Let $r=d(p,q)$ and $K=B_{r+2\delta}(p)$. Then $K$ is compact and purely 1-unrectifiable, so by Theorem \ref{thm:spu} there exists $g\in\ball{\lip(K)}$ such that $g(p)=0$ and $g(x)\geq d(p,x)-\delta$ for every $x\in K$. Now define $h:M\rightarrow\RR$ by
$$
h(x) = \begin{cases}
g(x) \wedge r &\text{, if } x\in K \\
r &\text{, if } x\notin K
\end{cases}
$$
Note that $h(x)=r$ if $d(x,p)\geq r+\delta$, so it is easy to check that $h$ is 1-Lipschitz, locally flat, and flat at infinity. Clearly $h(q)-h(p)\geq r-\delta$ so letting $f=h-h(0)$ provides a map in $\ball{\lip_0(M)}$ which separates $p$ and $q$ as required.

\ref{e:lip_spu}$\Rightarrow$\ref{e:lip_predual}: This is contained in \cite[Theorem 4.38]{Weaver2}. An earlier, different proof of the proper case is given in \cite{Dalet_2015_2} - the result is only explicitly stated for the particular case where $M$ is countable, but the argument is valid in general; see Lemmas 2.4 and 2.5 and the first paragraph of the proof of Theorem 2.1 therein.

\ref{e:lip_predual}$\Rightarrow$\ref{e:dual} is trivial.

\ref{e:dual}$\Rightarrow$\ref{e:rnp}: $M$ is separable, hence so is $\lipfree{M}$, and any separable dual Banach space has the RNP (see e.g. \cite[Corollary~2.15]{Pisier}).

\ref{e:lip_predual}$\Rightarrow$\ref{e:schur}: This is implied by \cite[Proposition 8]{Petitjean_2017} (note that $\lip_0(M)$ in that result's statement refers to the space of uniformly locally flat functions, which includes our little Lipschitz space in both the compact and proper settings).

\ref{e:rnp}$\Rightarrow$\ref{e:no_L1} and \ref{e:schur}$\Rightarrow$\ref{e:no_L1} follow from the fact that the Radon-Nikod\'ym and the Schur properties are hereditary and preserved by isomorphisms, and $L_1$ fails both of them.

\ref{e:no_L1}$\Rightarrow$\ref{e:p1u}: As explained above, this follows from Godard's theorem \cite[Corollary 3.4]{Godard_2010}.
\end{proof}

The equivalent conditions in Theorem \ref{th:seven_equivalences} imply a strong $\ell_1$-like behavior of $\lipfree{M}$. In fact, a stronger quantitative version of the Schur property called the \textit{1-strong Schur property} is fulfilled in this case (see \cite{Petitjean_2017} for the definition and Proposition 17 therein for the immediate proof). Let us point out that
$\Free(M)$ is also \weak$^\ast$ asymptotically uniformly convex with power type 1 modulus in this case (see e.g. \cite[Proposition~4.4.2]{ColinThesis}).

There are several other Banach space properties, like the Krein-Milman property, that lie between the RNP or the Schur property and the non-containment of $L_1$ and are therefore equivalent to them under the hypothesis of Theorem \ref{th:seven_equivalences}. We will see later that this equivalence holds in a more general case, so we withhold the discussion of these properties until Section \ref{section:theoremC}.

\subsection{Preduals of \texorpdfstring{$\lipfree{M}$}{F(M)} with additional structure}

For proper $M$, Theorem \ref{th:seven_equivalences} shows that $\lip_0(M)$ is a predual of $\lipfree{M}$ whenever there is one. This doesn't preclude the possibility of there being other preduals.
In fact, $\lip_0(M)$ is never a unique predual unless it is finite-dimensional. Indeed, $\lip_0(M)$ embeds almost-isometrically into $c_0$ by \cite[Lemma 3.9]{Dalet_2015_2}, which implies that it is an M-embedded Banach space (see \cite[Section III.1]{HWW} for the notion and the proof of that fact). Thus \cite[Proposition III.2.10]{HWW} proves that $\lip_0(M)$ is not a unique predual. However, $\lipfree{M}=\dual{\lip_0(M)}$ is then L-embedded and so \cite[Proposition IV.1.9]{HWW} shows that $\lip_0(M)$ is, up to isometric isomorphism, the unique predual of $\lipfree{M}$ that is M-embedded.

We shall now prove that $\lip_0(M)$ is also the unique predual that satisfies the constraint of having an inherited lattice structure. Recall that the spaces $\Lip_0(M)$ and $\Lip(M)$ are vector lattices under the operations $\vee$ and $\wedge$ of pointwise maximum and minimum, with $\Lip(M)=\Lip_0(M)+\lspan\set{\mathsf{1}_M}$ where $\mathsf{1}_M$ denotes the function on $M$ that takes the constant value $1$. A similar statement holds for $\lip_0(M)$ and $\lip(M)$.

We will say that $Y$ is a \textit{linear sublattice} of $\Lip_0(M)$ or $\Lip(M)$ if it is a linear subspace such that $f\vee g\in Y$ (and thus also $f\wedge g\in Y$) whenever $f,g\in Y$. Following \cite{Weaver2}, we say that a linear sublattice of $\Lip_0(M)$ is \textit{shiftable} if $f\vee (g-c\cdot \mathsf{1}_M)\in Y$ whenever $f,g\in Y$ and $c\geq 0$. The artificial-looking condition $c\geq 0$ ensures that the resulting function still vanishes at $0$. It is straightforward to check that $Y$ is a shiftable linear sublattice of $\Lip_0(M)$ if and only if $Y+\lspan\set{\mathsf{1}_M}$ is a linear sublattice of $\Lip(M)$. Thus, shiftable linear sublattices are invariant with respect to a change of base point in $M$. This is not true for all sublattices: consider e.g. the one-dimensional space generated by the function $x\mapsto d(x,0)$.

Let us say that a real-valued function $f$ defined on a metric space is \textit{coercive} if $\abs{f(x)}\rightarrow\infty$ as $d(x,0)\rightarrow\infty$. Notice that coercivity does not depend on the choice of base point, and that any function defined on a bounded metric space is coercive by vacuity. Thus, we will only need the next lemma for the unbounded case:

\begin{lemma}
\label{lm:hb_coercive}
Let $M$ be a proper metric space and let $Y$ be a closed linear sublattice of $\Lip_0(M)$. Suppose that $\ball{Y}$ is \weaks-dense in $\ball{\Lip_0(M)}$. Then $Y$ contains a positive coercive function.
\end{lemma}

\begin{proof}
Let us first verify the following simple claim:
\begin{claim*} 
Given a positive $f\in\ball{\Lip_0(M)}$, $r>0$ and $\varepsilon>0$, there is a positive $g\in\ball{Y}$ such that $\abs{f(x)-g(x)}\leq\varepsilon$ for every $x\in B_r(0)$.
\end{claim*}
\noindent Indeed, we can find a finite $\varepsilon/3$-net $A$ in $B_r(0)$ because it is a compact set. Since $\ball{Y}$ is \weaks-dense in $\ball{\Lip_0(M)}$, there is $h\in\ball{Y}$ such that $\abs{f(y)-h(y)}\leq\varepsilon/3$ for every $y\in A$. Now if $x\in B_r(0)$, find $y\in A$ with $d(x,y)\leq\varepsilon/3$ and we have
$$
\abs{f(x)-h(x)}\leq\abs{f(x)-f(y)}+\abs{f(y)-h(y)}+\abs{h(y)-h(x)}\leq\varepsilon .
$$
Then $g=h\vee 0$ satisfies the requirements.

Now use the Claim to obtain positive functions $f_n\in\ball{Y}$ for every $n\in\NN$ such that $f_n(x)\geq d(x,0)-1$ for every $x\in B_{4^n}(0)$, and let $f=\sum_{n=1}^\infty f_n/2^n$. Then $f\in Y$ because $Y$ is closed, and if $x\in M$ is such that $4^k<d(x,0)\leq 4^{k+1}$ for some $k\in\NN$ then
$$
f(x)\geq \sum_{n=k+1}^\infty \frac{f_n(x)}{2^n}\geq \sum_{n=k+1}^\infty \frac{d(x,0)-1}{2^n} > \frac{4^k-1}{2^k} > 2^k-1 .
$$
This shows that $f$ is coercive.
\end{proof}

For the next result we need to introduce some notation. Let
$$
\widetilde{M} = \set{(x,y)\in M\times M: x\neq y}
$$
be the set of pairs of different points of $M$, with the topology inherited from $M\times M$. For $(x,y)\in\widetilde{M}$ denote
$$
\mol{xy}=\frac{\delta(x)-\delta(y)}{d(x,y)}
$$
which is a norm $1$ element of $\lipfree{M}$, usually called the \textit{elementary molecule} determined by $x$ and $y$. The \textit{de Leeuw map} \cite{deLeeuw_1961} is the mapping $\Phi$ that takes a function $f\in\Lip_0(M)$ to the function $\Phi f:\widetilde{M}\rightarrow\RR$ defined by
$$
\Phi f(x,y)=\duality{\mol{xy},f}=\frac{f(x)-f(y)}{d(x,y)} .
$$
Clearly $\Phi f$ is continuous and bounded, with $\norm{\Phi f}_\infty=\lipnorm{f}$, and so it can be identified with its continuous extension to $\beta\widetilde{M}$, the Stone-\v{C}ech compactification of $\widetilde{M}$. Thus we may regard $\Phi$ as a linear isometry from $\Lip_0(M)$ into $C(\beta\widetilde{M})$.

With the required notation in place, we may state and prove the following variation of \cite[Theorem 3.43]{Weaver2}:

\begin{lemma}
\label{lm:weaver_lemma}
Let $M$ be a proper metric space and let $Y$ be a shiftable linear sublattice of $\Lip_0(M)$ that contains a coercive function. Let $\zeta\in\beta\widetilde{M}$ and assume that there exists $g\in Y$ such that $\Phi g(\zeta)\neq 0$. Then $\zeta\in\widetilde{M}$ if and only if $\Phi f_i(\zeta)\rightarrow\Phi f(\zeta)$ for every bounded net $(f_i)$ in $Y$ that converges pointwise to $f\in Y$.
\end{lemma}

\begin{proof}
The forward implication is obvious, since $\zeta=(x,y)\in\widetilde{M}$ implies $\Phi f_i(\zeta)=\duality{\mol{xy},f_i}\rightarrow\duality{\mol{xy},f}=\Phi f(\zeta)$ whenever $f_i\wsconv f$ in $\Lip_0(M)$. For the backward implication, choose any $\zeta\in\beta\widetilde{M}\setminus\widetilde{M}$ and we will show that there is a bounded sequence $(f_n)$ in $Y$ that converges pointwise to $f$ but fails $\Phi f_n(\zeta)\rightarrow\Phi f(\zeta)$. We will follow the proof of \cite[Theorem 3.43]{Weaver2} closely.

Let $(x_i,y_i)$ be a net in $\widetilde{M}$ that converges to $\zeta$; we may assume that $x_i$ and $y_i$ converge to elements $\xi$ and $\eta$ of $\beta M$. We may also take $g\geq 0$: since $Y$ is a sublattice, $g^{+}=g\vee 0$ and $g^{-}=(-g)\vee 0$ belong to $Y$ and are non-negative, and at least one of $\Phi g^{+}(\zeta)$, $\Phi g^{-}(\zeta)$ must be different from $0$.
Write $d(\zeta)=\lim_i d(x_i,y_i)\in [0,\infty]$, and notice that $\abs{g(\xi)-g(\eta)}=\abs{\Phi g(\zeta)}\cdot d(\zeta)$. We now distinguish three cases:
\smallskip

\textit{Case 1:} $d(\zeta)=0$. Then $g(\xi)=g(\eta)$; denote this value by $a$. Suppose first that $a<\infty$, let
$$
g_n = \pare{g \wedge (1+\tfrac{1}{n})a\cdot \mathsf{1}_M} \vee \pare{1-\tfrac{1}{n}}a\cdot \mathsf{1}_M
$$
and $f_n=g_n-g_n(0)\cdot \mathsf{1}_M\in Y$. Then $\lipnorm{f_n}\leq\lipnorm{g}$ and $f_n(x)$ decreases to $0$ for every $x\in M$. However, for each $n$ we have $\abs{g(x_i)-a}<\frac{1}{n}$ and $\abs{g(y_i)-a}<\frac{1}{n}$ eventually, so $\Phi f_n(x_i,y_i)=\Phi g_n(x_i,y_i)=\Phi g(x_i,y_i)$ eventually and $\Phi f_n(\zeta)=\Phi g(\zeta)$. That is, $f_n\rightarrow 0$ pointwise but $\Phi f_n(\zeta)\rightarrow\Phi g(\zeta)\neq 0$.

If $a=\infty$, take $f_n=(g-n\cdot \mathsf{1}_M)\vee 0\in Y$ instead. Again, $\lipnorm{f_n}\leq\lipnorm{g}$ and $f_n\rightarrow 0$ pointwise, yet $g(x_i),g(y_i)>n$ eventually and so $\Phi f_n(\zeta)=\Phi g(\zeta)$ for each $n$.
\smallskip

\textit{Case 2:} $d(\zeta)=\infty$. Take $f_n=(g-n\cdot \mathsf{1}_M)\vee 0\in Y$ again. In this case,
$$
\lim_i\abs{\Phi f_n(x_i,y_i)-\Phi g(x_i,y_i)}\leq\lim_i\frac{2n}{d(x_i,y_i)}=0
$$
since $\norm{f_n-g}_\infty\leq n$, and so $\Phi f_n(\zeta)=\Phi g(\zeta)$ for fixed $n$.
\smallskip

\textit{Case 3:} $0<d(\zeta)<\infty$. We will show that we may assume that $g(\xi)=g(\eta)=\infty$, so we may then take $f_n=(g-n\cdot \mathsf{1}_M)\vee 0$ yet again and apply the argument from Case 1.

Indeed, assume otherwise. Then $g(\xi)$ and $g(\eta)$ are both finite since $\abs{g(\xi)-g(\eta)}=\abs{\Phi g(\zeta)}\cdot d(\zeta)<\infty$. Now, observe that $\xi$ and $\eta$ cannot both belong to $M$ as $\zeta\notin\widetilde{M}$. If, say, $\xi\notin M$, then $d(\xi,0)=\infty$ because $M$ is proper (if $d(\xi,0)<r$, then the compactness of $B_r(0)$ would imply $\xi\in M$), and $d(\zeta)<\infty$ implies that $d(\eta,0)=\infty$ as well. By hypothesis, there is a coercive function $v\in Y$, and we may assume $v\geq 0$ by replacing it with $\abs{v}\in Y$. Thus $v(\xi)=v(\eta)=\infty$. If $\Phi v(\zeta)\neq 0$, then take $v$ instead of $g$ to get the desired contradiction. Otherwise, use the function $g+v$.
\end{proof}

We may now prove the desired uniqueness result.

\begin{theorem}
\label{th:lattice_predual}
Let $M$ be a proper metric space. If $Y$ is a closed shiftable linear sublattice of $\Lip_0(M)$ such that $\dual{Y}=\lipfree{M}$, then $Y=\lip_0(M)$.
\end{theorem}

\begin{proof}
We consider the weak$^\ast$ topology in $\lipfree{M}$ induced by $Y$. Let $\zeta\in\beta\widetilde{M}$, and let $(x_i,y_i)$ be a net in $\widetilde{M}$ that converges to $\zeta$. Then $(\mol{x_iy_i})$ is a net in the unit ball of $\lipfree{M}$, which is a \weaks-compact set, so we can replace $(\mol{x_iy_i})$ with a subnet such that $\mol{x_iy_i}\wsconv\mu$ for some $\mu\in\lipfree{M}$. This means that
$$
\duality{\mu,f} = \lim_i\duality{\mol{x_iy_i},f} = \lim_i\Phi f(x_i,y_i) = \Phi f(\zeta)
$$
for every $f\in Y$. Now let $(f_i)$ be a bounded net in $Y$ that converges pointwise to $f\in Y$. Then $f_i\wsconv f$ in the weak$^\ast$ topology of $\Lip_0(M)$, and so $\duality{\mu,f_i}\rightarrow\duality{\mu,f}$, that is, $\Phi f_i(\zeta)\rightarrow\Phi f(\zeta)$.

We have thus shown that $Y$ satisfies the condition in Lemma \ref{lm:weaver_lemma}. Since it also contains a coercive function by Lemma \ref{lm:hb_coercive}, it follows that either $\zeta\in\widetilde{M}$ or $\mu|_Y=0$. Therefore $\Phi f(\zeta)=0$ for every $f\in Y$ and every $\zeta\in\beta\widetilde{M}\setminus\widetilde{M}$. By \cite[Proposition 4.20]{Weaver2}, this is equivalent to $Y\subset\lip_0(M)$. But $\lip_0(M)$ is also a predual of $\lipfree{M}$ by Theorem \ref{th:seven_equivalences}, so $Y=\lip_0(M)$.
\end{proof}


\section{The Radon-Nikod\'ym property in Lipschitz-free spaces}
\label{section:rnp}

Very recently, a method of ``compact reduction'' was introduced in \cite{ANPP_2020, Gartland_2020} to prove some properties of Lipschitz-free spaces by considering only the compact case. Let us say informally that a Banach space property (P) is compactly determined in Lipschitz-free spaces if for every metric space $M$, the free space $\Free(M)$ has (P) whenever $\Free(K)$ has (P) for every compact subset $K \subset M$. 
For instance, it is proved in \cite[Corollary~2.6]{ANPP_2020} that the Schur property satisfies such a statement (see also \cite[Corollary~2.9]{ANPP_2020} for $\ell_1$-saturation).

In order to extend the equivalences in Theorem \ref{th:seven_equivalences} to the noncompact case, it would be enough to show that the corresponding properties are compactly determined. Unfortunately, the property ``$\Free(M)$ is (isometrically) a dual space'' is not determined by compact subsets of $M$. There are indeed examples of uniformly discrete metric spaces $M$ for which $\Free(M)$ is not isometric to any dual Banach space (see \cite[Example 5.8]{GPPR_2018} and \cite[Remark~6.4]{BLPP_2019}), although any compact subset $K \subset M$ is finite so that $\Free(K)$ is finite-dimensional. However, we will prove in this section that the RNP is compactly determined in Lipschitz-free spaces, and its equivalence with the Schur property will follow by compact reduction.

\subsection{Compact determination of the RNP}

The Radon-Nikod\'ym property admits many equivalent formulations; see e.g. \cite[Section VII.6]{DiestelUhl} and \cite[Theorem~2.9]{Pisier}. Our arguments in this section will be based on its characterization in terms of martingales in Bochner spaces. Let us recall the notions involved in it.

Let $(\Omega,\A,\PP)$ be a probability space and let $X$ be a Banach space. We denote by $L_1(\Omega,\A,\PP; X)$ the space of (equivalence classes of) Bochner measurable functions $f : \Omega \to X$ such that $\int_{\Omega} \|f\|_X \, d\PP < \infty$ equipped as usual with the norm
$$ \|f\|_{L_1(\Omega,\A,\PP; X)} = \E(\|f\|_X) = \int_{\Omega} \|f\|_X \, d\PP.$$
In the sequel, we will suppress notation and write simply $L_1(X)$, or $L_1(\A;X)$ when the $\sigma$-algebra $\A$ needs to be emphasized.

We recall that a sequence $(M_n)_{n=0}^\infty$ in $L_1(\A; X)$ is called a \textit{martingale} if there exists an increasing sequence $(\A_n)_{n=0}^\infty$ of $\sigma$-subalgebras of $\A$ (called a \textit{filtration}) such that for each $n \geq 0$, $M_n$ is $\A_n$-measurable and satisfies
$$ M_n = \E^{\A_n}(M_{n+1}), $$
where $\E^{\A_n}$ denotes the $X$-valued conditional expectation (see e.g. \cite[Section~1.2]{Pisier}).
We say moreover that $(M_n)_{n=0}^\infty$ is \textit{uniformly integrable} if the sequence of non-negative random variables $(\|M_n\|_X)_{n=0}^\infty$ is uniformly integrable. More precisely, this means that $(\|M_n\|_X)_{n=0}^\infty$ is bounded in $L_1(\RR)$ and that, for any $\eps>0$, there is a $\delta>0$ such that $$ \forall A \in \A, \quad \PP(A) < \delta \quad\implies\quad \sup_{n \geq 0} \int_A \|M_n\|_X \, d\PP < \eps.$$

Whenever $T: X \to Y$ is a bounded linear operator, we get a well-defined bounded linear operator $\widetilde{T}: L_1(X) \to L_1(Y)$ defined by $\widetilde{T}(f) = T \circ f$ with $\|\widetilde{T}\| = \|T\|$. An elementary but important fact is that $\widetilde{T}$ commutes with any conditional expectation operator $\mathbb{E}^{\A_n}$. In the sequel, we will abuse notation and denote $\widetilde{T}$ by $T$.

A Banach space $X$ has the RNP if and only if every uniformly integrable $X$-valued martingale converges in $L_1(X)$ (see \cite[Theorem~2.9]{Pisier}). We will use this characterization in order to prove that the RNP is compactly determined in Lipschitz-free spaces. Our way to the proof will be similar to the method used in \cite{ANPP_2020}.

For the remainder of this section, fix a complete metric space $(M,d)$ and a probability space $(\Omega,\A,\PP)$.

\begin{definition} \label{def:KPtight}
Inspired by the terminology from \cite{ANPP_2020}, we say that a collection $W\subset L_1(\lipfree{M})$ of random variables $F:(\Omega,\A,\PP) \to \lipfree{M}$
has the \textit{mean Kalton property} if, for every $\eps,r > 0$, there exists a finite set $E \subset M$ such that
$$
d(F,L_1(\lipfree{[E]_r})) \leq \eps \quad\text{for all }F\in W,
$$
and we say that $W$ is \textit{mean-tight} if, for every $\eps>0$, there exist a compact $K \subset M$ such that 
$$
d(F,L_1(\lipfree{K})) \leq \eps \quad\text{for all }F\in W.
$$
\end{definition}
It is clear that if $W$ is mean-tight then it has the mean Kalton property. We will now show that the converse is actually also true. 

\begin{proposition} \label{thm:KPtight}
Let $W\subset L_1(\lipfree{M})$ be a collection of random variables with the mean Kalton property. Then $W$ is mean-tight. More precisely, for every $\eps>0$ there exist a compact $K \subset M$ and a mapping $T : W \to L_1(\lipfree{K})$ such that
\begin{itemize}
\item $\E(\norm{F-T(F)}) \leq \eps$ for every $F\in W$, and
\item there is a sequence of bounded linear operators $T_n: \lipfree{M} \to \lipfree{M}$, $n\in\NN$ such that
$$
\lim_{n\to\infty}\sup_{F\in W} \E(\norm{T_n(F)-T(F)})=0 .
$$
\end{itemize}
\end{proposition}

\begin{proof}
We follow the proof of \cite[Theorem~3.2]{ANPP_2020} closely. Suppose first that $M$ is bounded, let $R=\diam(M)$ and fix $\eps\in (0,\frac{1}{3})$. Set $\eps_0=\eps$ and $\delta_0=R$, and for $n\geq 1$ let $\eps_n=2^{-n}\eps$ and $\delta_n=R\cdot (\frac{1}{\eps_n}-2)^{-1}$. Let also $K_0=M$ and $S_0$ be the identity operator on $\lipfree{M}$. We will now construct, for $n\geq 1$, finite subsets $E_n\subset M$ containing $0$, closed subsets $K_n\subset M$ also containing $0$, and operators $S_n:\lipfree{M}\rightarrow\lipfree{M}$ such that
\begin{enumerate}[label={(\alph*)}]
\item $K_n = K_{n-1} \cap [E_n]_{2\delta_n}$,
\item the operators $S_n$ commute with each other,
\item $T_n(\lipfree{M})\subset\lipfree{K_n}$, and
\item $\E(\norm{T_{n-1}(F)-T_n(F)})\leq\eps_n$ for all $F\in W$,
\end{enumerate}
where $T_n=S_n\circ S_{n-1}\circ\ldots\circ S_0$.

We proceed by induction. Let $n\geq 1$ and suppose that $K_{n-1}$ and $S_{n-1}$ have already been constructed. By the mean Kalton property, we may find a finite set $E_n\subset M$ such that $0\in E_n$ and
\begin{equation}
\label{eq:TkMn_has_KP}
\norm{T_{n-1}}\cdot \sup_{F\in W}d(F,L_1(\lipfree{[E_n]_{\delta_n}})) < \eps_n^2.
\end{equation}
Let $K_n$ be as in (a). Now consider the function $h_n:M\to\RR$ given by
$$
h_n(x) = 0 \vee (1-\delta_n^{-1}\cdot d(x,[E_n]_{\delta_n}))
$$
for $x\in M$, and define the operator $S_n$ by
$$
\duality{S_n\mu,f} = \duality{\mu,f\cdot h_n}
$$
for $\mu\in\lipfree{M}$ and $f\in\Lip_0(M)$. Clearly $0\leq h_n\leq 1$, $\lipnorm{h_n}\leq\delta_n^{-1}$ and $\supp(h_n)\subset [E_n]_{2\delta_n}$, so by \cite[Proposition 2.4]{APPP_2020} we have $\norm{S_n} \leq 1+R/\delta_n = \eps_n^{-1}-1$. It is also clear that the restriction of $S_n$ to $\lipfree{[E_n]_{\delta_n}}$ is the identity, and that (b) holds for all $S_n$ defined in this way. Moreover, by formula (3) in \cite{APPP_2020} and inductive hypothesis (c) we have
$$
T_n(\lipfree{M}) = S_n(T_{n-1}(\lipfree{M})) \subset S_n(\lipfree{K_{n-1}}) \subset \lipfree{K_{n-1}\cap [E_n]_{2\delta_n}} = \lipfree{K_n}
$$
so (c) holds.
Finally let us check (d). Fix $F\in W$, then by \eqref{eq:TkMn_has_KP} there is $G\in L_1(\lipfree{[E_n]_{\delta_n}})$ such that $\E(\norm{T_{n-1}(F)-T_{n-1}(G)})<\eps_n^2$, and we have $T_n(G)=T_{n-1}(S_n(G))=T_{n-1}(G)$ by (b). So
\begin{align*}
\E\pare{\norm{T_{n-1}(F)-T_n(F)}} &\leq \E\pare{\norm{T_{n-1}(F)-T_{n-1}(G)}} + \E\pare{\norm{T_n(G)-T_n(F)}} \\
&\leq (1+\norm{S_n})\E\pare{\norm{T_{n-1}(F)-T_{n-1}(G)}} \leq \eps_n .
\end{align*}
The construction is thus complete.

Now let $K=\bigcap_{n=1}^\infty K_n$. Since $\delta_n\rightarrow 0$, (a) implies that $K$ is totally bounded. Then, since $K$ is closed and $M$ is complete, $K$ is compact. 
For every $F\in W$, it follows from (d) that the sequence $(T_n(F))$ is Cauchy so it converges to some element of $L_1(\lipfree{M})$. Denote the limit by $T(F)$. This gives us a map
$$
T: W \longrightarrow \bigcap_{n=1}^\infty L_1(\lipfree{K_n}) = L_1(\lipfree{K}) ,
$$
where the equality holds by Pettis' measurability theorem (see~\cite[p. 42]{DiestelUhl}) and the fact that $\bigcap_n\lipfree{K_n}=\lipfree{K}$ (see \cite[Theorem 2.1]{APPP_2020}).
Finally, notice that
$$
\E\pare{\norm{F-T_n(F)}} \leq \sum_{i=1}^n \E\pare{\norm{T_{i-1}(F)-T_i(F)}} < \eps
$$
by (d), and therefore $\E\pare{\norm{F-T(F)}}\leq\eps$. Similarly $\E\pare{\norm{T_n(F)-T(F)}}\leq\eps_n$ for every $n$. This completes the proof of the bounded case.

If $M$ is unbounded, then we simply replace the initial construction step as follows. By assumption we can find a finite set $E_0\subset M$ containing $0$ such that
$$d(F,L_1(\lipfree{[E_0]_1}))<\frac{\varepsilon}{4}$$
for every $F\in W$. Set $R=2(\diam(E_0)+1)$ and $K_0=B_R(0)$, consider the function $h_0:M\to\RR$ given by
$$
h_0(x)= 0 \vee (1 \wedge (2-2R^{-1}d(x,0)))
$$
for $x\in M$, and let $S_0:\lipfree{M}\to\lipfree{K_0}$ be defined by $\duality{S_0\mu,f}=\duality{\mu,f\cdot h_0}$ for $f\in\Lip_0(M)$. Then, similarly as above, we have $\lipnorm{h_0}\leq 2R^{-1}$, $\norm{S_0}\leq 3$, and thus
$$
\E\pare{\norm{F-S_0(F)}} \leq (1+\norm{S_0})\frac{\varepsilon}{4} \leq \varepsilon
$$
for every $F\in W$. We can then continue the inductive construction as above.
\end{proof}

Next, we obtain a probabilistic version of Kalton's lemma \cite[Lemma 4.5]{Kalton_2004} for martingales in place of weakly null sequences. 
This is the main new technical tool needed to prove compact determination of the RNP.

\begin{proposition} \label{thm:KP}
Suppose that $M$ is bounded. Then every $L_1(\lipfree{M})$-bounded martingale has the mean Kalton property.
\end{proposition}

\begin{proof}
Assume towards a contradiction that there exist $r>0$, $\eps>0$, and an $L_1(\lipfree{M})$-bounded martingale $(\tilde{M}_n)_{n=0}^\infty$ adapted to a filtration $(\tilde{\A}_n)_{n=0}^\infty$ such that for every finite subset $E \subset M$, there exists $n \in \NN$ with 
$$d(\tilde{M}_n,L_{1}(\lipfree{[E]_r})) = d(\tilde{M}_n,L_{1}(\tilde{\A}_n;\lipfree{[E]_r})) > 4\eps.$$
First we use a standard approximation technique to replace $(\tilde{M}_n)_{n=0}^\infty$ with a martingale $(M_n)_{n=0}^\infty$ adapted to a filtration $(\A_n)_{n=0}^\infty$ where each $\A_n$ is finite. Here is the technique: For each $n \geq 0$, choose an $\tilde{\A}_n$-measurable simple function $G_n$ such that $\|\tilde{M}_n-G_n\|_{L_1} < \eps$.
Let $\A_n$ be the (finite) $\sigma$-algebra generated by $G_0, G_1, \dots, G_n$. 
Then, using $\A_n \subset \tilde{\A}_n,\A_{n+1}$ and the tower property of conditional expectations, we can see that $(M_n)_{n=0}^\infty := (\E^{\A_n}(\tilde{M}_n))_{n=0}^\infty$ is an $L_1(\lipfree{M})$-bounded martingale adapted to $(\A_n)_{n=0}^\infty$. 
Now observe that $\|\tilde{M}_n-M_n\|_{L_1} < 2\eps$, which implies, for every finite $E \subset M$, there exists $n \in \NN$ with $$d(M_n,L_{1}(\lipfree{[E]_r})) = d(M_n,L_{1}(\A_n;\lipfree{[E]_r})) > 2\eps.$$

By using density of the finitely supported elements in $\lipfree{M}$, we may replace $(M_n)_n$ with a stochastic process $(\cl{M}_n)_n$ adapted to $(\A_n)_{n=0}^\infty$ satisfying
\begin{itemize}[itemsep=3pt]
\item $\cl{M}_n$ is constant on each atom of $\A_n$ and not just essentially constant,
\item $\cl{M}_n(\omega)$ is finitely supported for every $n \in \NN$ and $\omega \in \Omega$,
\item $(\cl{M}_n)_n$ is $L_1(\lipfree{M})$-bounded,
\item for every finite $E \subset M$, there exists $n \in \NN$ with 
$$d(\cl{M}_n,L_1(\A_n;\lipfree{[E]_r})) > \eps,$$ 
\item for every $n \leq i$, $\|\E^{\A_n}(\cl{M}_{i})-\cl{M}_n\|_{L_1} \leq 2^{-n} \wedge \eps$.
\end{itemize}
Notice that the last condition implies that $(\cl{M_n})_{n=0}^\infty$ is a quasi-martingale; see \cite[Remark~2.16]{Pisier} for more details.
\smallskip

Set $N_0 := 0$ and $E_0 := \bigcup_{n \leq N_0} \bigcup_{\omega \in \Omega} \supp(\cl{M}_n(\omega))$. It is clear that $E_0$ is finite. Then there exists $N_1 > N_0$ such that 
$d(\cl{M}_{N_1},L_1(\A_{N_1};\lipfree{[E_0]_r})) > \eps $.
We let $E_1 := \bigcup_{n \leq N_1} \bigcup_{\omega \in \Omega} \supp(\cl{M}_n(\omega))$, which is finite as well. Then there exists $N_2 > N_1$ such that $d(\cl{M}_{N_2},L_1(\A_{N_2};\lipfree{[E_1]_r})) > \eps$. Continuing in this way, we will get an increasing sequence of integers $(N_i)_{i=1}^{\infty} \subset \NN$ and a sequence  of finite sets  $E_0 \subset E_1 \subset E_2 \subset \dots \subset  M$ such that, setting $M'_i := \cl{M}_{N_i}$,
\begin{itemize}[itemsep=3pt]
\item $\supp(M'_i(\omega)) \subset E_i$ for every $\omega \in \Omega$,
\item $(M'_i)_i$ is $L_1(\lipfree{M})$-bounded,
\item $d(M'_{i+1},L_1(\A_{N_{i+1}};\lipfree{[E_i]_r})) > \eps$,
\item for all $n \leq i$, $\|\E^{\A_{N_n}}(M'_{i})-M'_n\|_{L_1} \leq 2^{-N_n} \leq 2^{-n}$. 
\end{itemize}
By the Hahn--Banach theorem, for each $i \geq 1$ there is $\tilde{f}_{i} \in L_{\infty}(\A_{N_{i}};\Lip_0(M)) = L_1(\A_{N_{i}};\lipfree{M})^\ast$ such that
\begin{itemize}[itemsep=3pt]
\item $\|\tilde{f}_{i}\|_{L_{\infty}(\Lip_0(M))} \leq 1$,
\item $\tilde{f}_{i}(\omega)$ vanishes on $[E_{i-1}]_r$ for every $\omega \in \Omega$,
\item $\E\left(\duality{\tilde{f}_i,M'_i}\right) > \eps$. 
\end{itemize}
For each $\omega \in \Omega$, using the McShane--Whitney extension theorem (see \cite[Theorem~1.33]{Weaver2} for instance) we obtain a Lipschitz map $f_{i}(\omega) : M \to \RR$ so that
\begin{itemize}[itemsep=3pt]
\item $f_{i}(\omega)$ vanishes on $[E_{i-1}]_r$,
\item $f_{i}(\omega)$ agrees with $\tilde{f}_{i}(\omega)$ on $E_{i}$,
\item $\supp(f_{i}(\omega)) \subset [E_{i}]_r$,
\item $\lipnorm{f_{i}(\omega)} \leq C \lipnorm{\tilde{f}_{i}(\omega)} \leq C$, 
\end{itemize}
where $C= 1 \vee (\rad(M)/r)$.
Moreover, since $\A_{N_i}$ is finite, the above can be done so that $f_i$ is $\A_{N_i}$-measurable.
Notice that the functions $(f_i(\omega))_i$ have pairwise disjoint supports, therefore
$$
\lipnorm{\sum_{n=1}^i f_n(\omega)} = \lipnorm{\sum_{n=1}^i f_n^+(\omega) - \sum_{n=1}^i f_n^-(\omega)} = \lipnorm{\bigvee_{n=1}^i f_n^+(\omega) - \bigvee_{n=1}^i f_n^-(\omega)} \leq 2C \,.
$$
Here we denote $f_n^+(\omega) = 0 \vee f_n(\omega)$ and $f_n^-(\omega) = 0 \vee (-f_n(\omega))$ pointwise. 
Thus, we have for every $i\geq 1$
\begin{align*}
2 C \|M'_{i}\|_{L_1} &\geq \E \Big(  \duality{\sum_{n=1}^i f_n,M'_i} \Big) \\
&=  \E \Big(  \sum_{n=1}^i  \duality{f_n,\E^{\A_{N_{n}}}(M'_{i})} \Big) \\
&\geq \E \Big(  \sum_{n=1}^i \duality{f_n,M'_n}\Big)
- C \sum_{n=1}^i \|\E^{\A_{N_{n}}}(M'_{i}) - M_{n}'\|_{L_1}  \\
&\geq \E \Big(  \sum_{n=1}^i \duality{f_n,M'_n} \Big) - C \\
&> i\cdot\eps - C.
\end{align*}
Since $i \in \NN$ was arbitrary, $C<\infty$, and $\eps>0$, this contradicts our assumption that $\sup_i \|M'_i\|_{L_1} < \infty$.
\end{proof}

\begin{remark}
\label{rm:KP_unbounded}
Proposition~\ref{thm:KP} remains true for unbounded metric spaces $M$. This can be proved using results in the recent paper \cite{AACD_20} where a bounded metric space $\mathcal{B}$ and a linear isomorphism $P : \lipfree{M} \to \lipfree{\mathcal{B}}$ are constructed. It is clear (using \cite[Remark on p.~5]{Pisier}) that if $(M_n)_{n=0}^\infty\subset L_1(\lipfree{M})$ is a bounded martingale adapted to $(\A_n)_{n=0}^\infty$ then so is $(PM_n)_{n=0}^\infty \subset L_1(\lipfree{\mathcal{B}})$, therefore $(PM_n)_{n=0}^\infty$ is mean-tight by the bounded versions of Propositions~\ref{thm:KPtight} and \ref{thm:KP}.  Finally, we claim that if $(M_n)_{n=0}^\infty \subset L_1(\lipfree{\mathcal{B}})$ is mean-tight then $(P^{-1}M_n)_{n=0}^\infty \subset L_1(\lipfree{M})$ is also mean-tight, which implies the mean Kalton property.
This observation requires using the fact that $P^{-1}$ respects supports and compactness (see Proposition 4.2(ix) and Lemma 7.3 in \cite{AACD_20}). We do not know whether there is a direct argument for this claim with the mean Kalton property in place of mean-tightness.
\end{remark}

We may now deduce the compact determination of the RNP as a consequence of the two previous results.

\begin{corollary} \label{cor:CR-RNP}
    The Lipschitz-free space $\lipfree{M}$ has the Radon-Nikod\'ym property if and only if $\lipfree{K}$ has the Radon-Nikod\'ym property for every compact $K \subset M$.
\end{corollary}

\begin{proof}
    Since the RNP passes to subspaces, the ``only if" implication is immediate. We will prove the contrapositive of the ``if" direction.
    
    Let us first remark that we only need to consider the case where $M$ is bounded. Indeed, assume $\lipfree{M}$ fails to have the RNP. Thanks to \cite[Proposition 4.3]{Kalton_2004}, $\Free(M)$ is isomorphic to a subspace of $\big(\sum_{k \in \ZZ} \Free(M_k)\big)_{\ell_1}$ for certain bounded subsets $M_k$ of $M$. Since the RNP is stable through $\ell_1$-sums, isomorphisms, and passing to subspaces, $\Free(M_k)$ must fail the RNP for some bounded $M_k \subset M$. 
    
    By \cite[Theorem~2.9]{Pisier}, there exists a uniformly integrable martingale $(M_n)_{n=0}^\infty$ in $L_1(\lipfree{M})$ that does not converge in $L_1(\lipfree{M})$. So there exists some $\delta>0$ such that $\limsup\limits_{n,m \to \infty}\|M_n-M_m\|_{L_1(\lipfree{M})} > \delta$. The martingale $(M_n)_n$ has the mean Kalton property by Proposition~\ref{thm:KP}, so we may apply Proposition~\ref{thm:KPtight} with $\eps = \frac{\delta}{4}$ to get a compact $K\subset M$, operators $T_k:\lipfree{M}\to\lipfree{M}$ and a map $T:(M_n)_n\to L_1(\lipfree{K})$ such that
    \begin{equation}\label{eq:KP_1}
    \sup_{n\in\NN} \norm{M_n-T(M_n)}_{L_1(\lipfree{M})}\leq\frac{\delta}{4}
    \end{equation}
    and 
    \begin{equation}\label{eq:KP_2}
    \lim_{k\to\infty}\sup_{n\in\NN} \norm{T_k(M_n)-T(M_n)}_{L_1(\lipfree{M})}=0 .   
    \end{equation}
	
	Let us see that $(T(M_n))_{n=0}^\infty$ is a martingale. Let $(\A_n)$ be a filtration with respect to which $(M_n)$ is a martingale. Since $T_k$ commutes with $\mathbb{E}^{\A_n}$ for every $k \geq 1$ and $n \geq 0$, and since $T_k(M_n) \to T(M_n)$ for every $n \geq 0$, it follows that \[
	\mathbb{E}^{\A_n}[T(M_{n+1})] = \lim_{k\to\infty} T_k(\mathbb{E}^{\A_n}[M_{n+1}])=\lim_{k\to\infty} T_k(M_n)=T(M_n).
	\]
	
	Now notice that $(T(M_n))$ is $L_1(\lipfree{K})$-bounded by \eqref{eq:KP_1}, since $(M_n)$ is $L_1(\lipfree{M})$-bounded. Moreover, the sequence $(T_k(M_n))_{n=0}^\infty$ is uniformly integrable for each $k$ since $T_k$ is bounded. Condition \eqref{eq:KP_2} then implies that $(T(M_n))$ is an $L_1(\lipfree{K})$-uniformly integrable martingale. But we have
	$$\limsup_{n,m \to \infty} \|T(M_n)-T(M_m)\|_{L_1(\lipfree{K})} > \frac{\delta}{2}$$
	by \eqref{eq:KP_1}, so it cannot converge in $L_1(\lipfree{K})$. According to \cite[Theorem~2.9]{Pisier} again, this shows that $\lipfree{K}$ fails to have the RNP, completing the proof.
\end{proof}

\subsection{The structure of Lipschitz-free spaces over purely 1-unrectifiable metric spaces}
\label{section:theoremC}

In general Banach spaces we have
\begin{center}
Schur property $\Longrightarrow$ non-containment of $L_1$ $\Longleftarrow$ RNP 
\end{center}
but all implications absent in the diagram fail in general (either trivially or by deep examples due to Hagler \cite{Hagler1977studia} and Bourgain and Rosenthal \cite{BR1980Israel} of spaces with the Schur property failing the RNP). The situation is quite different for Lipschitz-free spaces, as the three properties above are in fact equivalent.

\begin{theorem} \label{thm:equivalences}
	Let $M$ be a  metric space. Then the following are equivalent:
	\begin{enumerate}[label={\upshape{(\roman*)}}]
		\item \label{p1u} The completion of $M$ is purely 1-unrectifiable,
		\item \label{rnp} $\lipfree{M}$ has the Radon-Nikod\'ym property,
		\item \label{kmp} $\lipfree{M}$ has the Krein-Milman property,
		\item \label{schur} $\lipfree{M}$ has the Schur property,
		\item \label{no_L1} $\lipfree{M}$ contains no isomorphic copy of $L_1$,
		\item \label{no_L1_ep} There exists $\eps>0$ such that $\lipfree{M}$ contains no $(1+\eps)$-isomorphic copy of $L_1$.
	\end{enumerate}
\end{theorem}

\begin{proof}
By the basic properties of free spaces, the free space over $M$ and the free space over the completion of $M$ are the same, so it suffices to assume $M$ is complete.

Property \ref{p1u} is compactly determined by basic measure theory, and \ref{rnp} and \ref{schur} are compactly determined by virtue of Corollary~\ref{cor:CR-RNP} and \cite[Corollary 2.6]{ANPP_2020}, respectively. Thus they are equivalent by Theorem~\ref{th:seven_equivalences}. Implications \ref{rnp}$\Rightarrow$\ref{kmp}$\Rightarrow$\ref{no_L1} are true in general Banach spaces, as the Krein-Milman property is hereditary and preserved by isomorphisms, and it fails to hold in $L_1$. Implication \ref{no_L1}$\Rightarrow$\ref{no_L1_ep} is trivial, and we will prove \ref{no_L1_ep}$\Rightarrow$\ref{p1u} by contrapositive.

Suppose $M$ is not purely 1-unrectifiable, and let $\gamma: K \to M$ be a curve fragment. By Lemma \ref{lem:Kirchheim}, for every $\eps>0$, there exists a positive measure subset $A \subset K$ and $c \in (0,\infty)$ such that $\gamma \restrict_A$ is a $(1+\eps)$-bi-Lipschitz embedding into $(M,c \cdot d)$, where $d$ is the metric on $M$. Whenever $A \subset \RR$ has positive measure, Godard's theorem \cite[proof of Corollary 3.4]{Godard_2010} implies $\lipfree{A}$ contains an isometric copy of $L_1$. Since a $(1+\eps)$-bi-Lipschitz embedding $A \hookrightarrow (M,c \cdot d)$ lifts to a $(1+\eps)$-linear isomorphic embedding $\lipfree{A} \hookrightarrow \lipfree{M,c \cdot d}$, the latter space contains a $(1+\eps)$-isomorphic copy of $L_1$. It is easy to check that $\lipfree{M,c \cdot d}$ is isometric to $\lipfree{M}$, and thus the conclusion follows.
\end{proof}

Federer~\cite[3.3.19-3.3.21]{Federer_GMT} and the works referenced within contain many examples of purely 1-unrectifiable metric spaces. We single out another particularly important one here: any snowflaked metric space $(M,d^\alpha)$ with $0 < \alpha < 1$ (or more generally, $(M,\omega \circ d)$ where $\omega$ is a nontrivial gauge \cite[p. 180]{Kalton_2004}) has purely 1-unrectifiable completion, and thus Theorem \ref{thm:equivalences} implies that $\Free(M,\omega\circ d)$ has the RNP.  Kalton~\cite[comment after Proposition 6.3]{Kalton_2004} proved this whenever $M$ is a separable dual Banach space. He also proved that $\Free(M,\omega\circ d)$ has the Schur property in general~\cite[Theorem~4.6]{Kalton_2004}, and even though he did not explicitly ask the question for the RNP in the general case, it has been around ever since. Theorem~\ref{thm:equivalences}
also yields that the Banach spaces enjoying the uniform lifting property must have the RNP as they embed into a free space with the Schur property~(see \cite[p. 189]{Kalton_2004}). Finally, we also obtain the following, previously unknown consequence: if $M = M_1 \cup M_2$ and $\lipfree{M_1}$ and $\lipfree{M_2}$ have the RNP (resp. Schur property), then $\lipfree{M}$ has the RNP (resp. Schur property).

\begin{remark}
\label{remark:more_equivalences}
The argument for equivalence of \ref{kmp} in Theorem \ref{thm:equivalences} can be generalized. Indeed, the conditions in the theorem are also equivalent to any other property of $\lipfree{M}$ that lies between non-containment of $L_1$ and either the RNP or the Schur property, such as the following:
\begin{itemize}
\item The \textit{point of continuity property} (PCP). Recall that a Banach space $X$ has the PCP provided every non-empty weakly closed and bounded subset admits a point of continuity of the identity map from the weak to the norm topology (see e.g. \cite[Section 4]{FLP_handbook}).
\item The \textit{uniform Kadec-Klee property} (UKK). Recall that $X$ is UKK if for each $\varepsilon> 0$ there exists $\delta>0$ such that every $\varepsilon$-separated weakly convergent sequence in the closed unit ball of $X$ converges to an element of norm less than $1- \delta$ \cite{Huff_1980}.
\end{itemize}
\end{remark}


\section{A Rectifiable-Connectedness Based Characterization of 1-Critical Sets}
\label{section:Whitney}
The first main result of this section is Theorem \ref{ThmE}, by way of Corollary \ref{cor:Whitney}. We recall that a compact metric space $M$ is called \textit{1-critical} if it supports a non-constant locally flat Lipschitz function, i.e. if $\lip_0(M)\neq \set{0}$. Compact metric spaces $M$ that are disconnected are trivially 1-critical, and because of this, many authors include in the definition that $M$ is connected. However, it is more convenient for us to consider disconnected metric spaces as 1-critical, and we adopt this convention.

\begin{reptheorem}{ThmE}[Corollary \ref{cor:Whitney}]
A compact metric space fails to be 1-critical if and only if it is transfinitely almost-rectifiably-connected.
\end{reptheorem}

Let us give an intuitive explanation of the meaning of transfinite almost-rectifiable-connectedness. In \cite[Chapter 8]{Weaver2} Weaver has defined, for a given metric space $(M,d)$, a pseudometric $d_\LL$ on $M$ by
$$
d_\LL(x,y) := \sup\set{ |f(y)-f(x)| \;:\; f \in \ball{\lip(M)} }
$$
and a metric space $M_{\LL}$ obtained by identifying points $x,y \in M$ with $d_\LL(x,y) = 0$. In actuality, Weaver's definitions of $\lip(M)$ and $(M_{\LL},d_\LL$) differ from the ones we give for general metric spaces. However, they agree when $M$ is compact, and we cite results from \cite[Chapter 8]{Weaver2} only in this case. Clearly, $M$ fails to be 1-critical if and only if $M_\LL$ is a single point. Thus, describing the distance $d_\LL$ in terms of the geometry of $(M,d)$ yields the desired geometric characterization of 1-criticality. We achieve this goal in the next subsection, where we define a transfinite sequence of spaces $M_{ur}^{(\alpha)}$ and show that $M_\LL$ is naturally identified with $M_{ur}^{(\omega_1)}$. The space $M_{ur}^{(0)}$ is simply $M$, and intuitively, each space $M_{ur}^{(\alpha+1)}$ is obtained by collapsing every curve fragment in $M_{ur}^{(\alpha)}$ down to an $\HH^1$-null set. This process must stabilize before the first uncountable ordinal $\omega_1$, and hence $M_{ur}^{(\omega_1)}$ has no curve fragments to collapse, i.e. it is purely 1-unrectifiable. Then we use Theorem \ref{ThmA} to conclude:

\begin{informalthm}[\ref{thm:ML=Mur}]
$M_\LL = M_{ur}^{(\omega_1)}$.
\end{informalthm}

\noindent See Theorem \ref{thm:ML=Mur} for a precise formulation. This theorem shows that $M$ fails to be 1-critical if and only if $M_{ur}^{(\omega_1)}$ is a single point. We are thus lead to the following definition: $M$ is \textit{transfinitely almost-rectifiably-connected} if $M_{ur}^{(\omega_1)}$ is a single point. We choose this terminology because the statement ``$M_{ur}^{(1)}$ is a single point" is almost equivalent to the statement ``$M$ is rectifiably-connected" (but not quite, see Example \ref{ex:topsinecurve}).

It is well-known (at least as early as \cite{Choquet_1944}) and easy to check that $\lambda(f(M)) = 0$ whenever $M$ is $\HH^1$-$\sigma$-finite and $f \in \lip(M)$. Hence if $M$ is also connected, then $f$ is constant and thus $M$ fails to be 1-critical. Our second main result in this section is Theorem \ref{ThmF} (by way of Theorem \ref{thm:bttree}), where we obtain a quantitative converse of this statement for bounded turning trees (see Definition \ref{def:bttree}) as a natural application of Theorem \ref{thm:ML=Mur}.

\begin{reptheorem}{ThmF}[Theorem \ref{thm:bttree}]
Let $(M,d)$ be a 1-bounded turning tree. Then for all $x,y \in M$,
$$d_{\LL}(x,y) = \inf\{\HH^1_\infty(A): [x,y] \setminus A \text{ is } \HH^1\text{-}\sigma\text{-finite}\}.$$
In particular, a bounded turning tree fails to be 1-critical if and only if each of its subarcs is $\HH^1$-$\sigma$-finite.
\end{reptheorem}

It was actually proved by Choquet in \cite{Choquet_1944} that $\lambda(f(M)) = 0$ for $\HH^1$-$\sigma$-finite $M$ under the weaker hypothesis that $f$ satisfies the pointwise flatness condition
$$\lim_{y \to x} \frac{|f(x)-f(y)|}{d(x,y)} = 0$$
for every $x \in M$. Thus, Theorem \ref{ThmF} shows that a bounded turning tree is 1-critical if and only if it supports a nonconstant pointwise flat function, reproving \cite[Theorem~2.2]{CKZ_2008}. Theorem \ref{ThmF} also generalizes a result of Norton \cite[Theorem~3]{Norton_1989} who proved that quasiarcs of Hausdorff dimension strictly larger than 1 are 1-critical.

In the final subsection, we provide examples of spaces $M$ for which the transfinite sequence $M = M_{ur}^{(0)} \to M_{ur}^{(1)} \to M_{ur}^{(2)} \to \dots$ stabilizes after one step and examples for which the sequence does not stabilize after one step. We also introduce curve-flat Lipschitz functions and use them as a tool for proving non-stabilization.

\subsection{\texorpdfstring{$\boldsymbol{M_\LL = M_{ur}^{(\omega_1)}}$}{ML=Mur-omega}}
\label{subsec:ML=Mur}

\begin{definition}
Let $(M,d)$ be a metric space. We define a pseudometric $d_{ur}$ on $M$ by
$$ d_{ur}(x,y) := \inf_K \lambda([\min(K),\max(K)] \setminus K) $$
where the infimum is over all compact $K \subset \RR$ such that there exists a 1-Lipschitz map $\gamma: K \to M$ with $\gamma(\min(K)) = x$ and $\gamma(\max(K)) = y$. After identifying any points $x,y$ with $d_{ur}(x,y) = 0$, we obtain a metric space $(M_{ur},d_{ur})$ and a canonical 1-Lipschitz surjection $q: M \to M_{ur}$. The surjection is 1-Lipschitz because for any $x,y \in M$, we may take $K = \{0,d(x,y)\}$, $\gamma(0) = x$, and $\gamma(d(x,y)) = y$, and thus $d(x,y)$ belongs to the set whose infimum equals $d_{ur}(x,y)$. When we wish to emphasize the domain $M$ of the map $q$, we will write $q^M$. Whenever $N$ is a second metric space and $f: M \to N$ is a 1-Lipschitz map, there is a canonically induced 1-Lipschitz map $f_{ur}: M_{ur} \to N_{ur}$ defined by $f_{ur}(q^M(x)) := q^N(f(x))$. It is easy to verify that this is a well-defined 1-Lipschitz map and that the functorial property $(f \circ g)_{ur} = f_{ur} \circ g_{ur}$ holds.
\end{definition}

\begin{remark} 
Let us explain an equivalent characterization of the pseudometric $d_{ur}$ induced by any isometric embedding $M \hookrightarrow X$ into a Banach space. It holds that
$$ d_{ur}(x,y) = \inf_{I,K,\gamma} \lambda(I \setminus \gamma^{-1}(K)) $$
where the infimum is over all compact intervals $I \subset \RR$, compact subsets $K \subset M$, and 1-Lipschitz curves $\gamma: I \to X$ with $\gamma(\min(I)) = x$ and $\gamma(\max(I)) = y$. The equality follows from the fact that for every compact $K \subset \RR$ and 1-Lipschitz map $\gamma: K \to M$, there exists a 1-Lipschitz extension $[\min(K),\max(K)] \to X$ by interpolating with line segments.

The pseudometric $d_{ur}$ admits yet another equivalent characterization in terms of curve-flat Lipschitz functions - see Definition \ref{def:curveflat} and Proposition \ref{prop:dur=dG}.

As will be proved in Proposition \ref{prop:collapse}, the map $q: M \to M_{ur}$ collapses every curve fragment $\gamma(K)$ down to an $\HH^1$-null set.
\end{remark}

\begin{example}[Subsets of $\RR$]
It is readily seen that if $M \subset \RR$ is compact, then $d_{ur}(x,y)=\left|\int_x^y \mathsf{1}_{\RR \setminus M}\,d\lambda \right|=\abs{f(x)-f(y)}$, where $f(x) := \int_0^x \mathsf{1}_{\RR \setminus M}\,d\lambda$. Since $\lambda(f(M))=0$, $M_{ur}$ is isometric to a $\lambda$-null subset of $\RR$ and hence is purely 1-unrectifiable.
\end{example}

\begin{example}[Topologist's Sine Curve] \label{ex:topsinecurve}
Obviously, if $(M,d)$ is rectifiably-connected, $d_{ur} \equiv 0$ and $M_{ur}$ is a single point. The converse statement is false, and the topologist's sine curve
$$(\{0\} \times [0,1]) \cup \{(x,\sin(\tfrac{1}{x})) \in \RR^2: x \in (0,1]\} \subset \RR^2$$
is a counterexample. For any $y \in \{(x,\sin(\frac{1}{x})) \in \RR^2: x \in (0,1]\}$, $z \in \{0\} \times [0,1]$, and $\eps > 0$, there is a rectifiable curve starting at $y$ and ending within a distance $\eps$ from $z$. However, the length of this curve necessarily goes to $\infty$ as $\eps \to 0$, and this prevents true rectifiable-connectedness.
\end{example}

\begin{proposition} \label{thm:isop1u}
Let $(M,d)$ be a metric space. Then $q: M \to M_{ur}$ is an isometry if and only if $M$ is purely 1-unrectifiable.
\end{proposition}

\begin{proof}
Assume $q: M \to M_{ur}$ is not an isometry. Then there exist $x,y \in M$ such that $d_{ur}(x,y) < d(x,y)$. Isometrically embed $M$ into a Banach space $X$. Then there exist a compact interval $I \subset \RR$, a compact subset $K \subset M$, and a 1-Lipschitz curve $\gamma: I \to X$ such that $\gamma(\min(I)) = x$, $\gamma(\max(I)) = y$, and $\lambda(I \setminus \gamma^{-1}(K)) < d(x,y)$. Then we have
\begin{align*}
d(x,y) \leq \HH^1(\gamma(I)) &= \HH^1(\gamma(I) \cap K) + \HH^1(\gamma(I) \setminus K) \\
&\leq \HH^1(\gamma(I) \cap K) + \lambda(I \setminus \gamma^{-1}(K)) \\
&< \HH^1(\gamma(I) \cap K) + d(x,y)
\end{align*}
where the second-to-last inequality follows from the fact that $\gamma$ is 1-Lipschitz. This shows $\HH^1(\gamma(I) \cap K) > 0$ and hence $M$ is not purely 1-unrectifiable.

Now assume $M$ is not purely 1-unrectifiable. Then there exists a compact $K' \subset \RR$ with $\lambda(K') > 0$ and a bi-Lipschitz embedding $\gamma: K' \to M$. By precomposing with a dilation, we may assume $\gamma$ is 1-Lipschitz and $\gamma^{-1}$ is $L$-Lipschitz for some $L \in [1,\infty)$. By Lebesgue's density theorem, there exist $t \in K'$ and $r > 0$ such that $\lambda([t,t+r] \setminus K') < \frac{r}{4L}$. Set $K := [t,t+r] \cap K'$. It must hold that $[t,t+\frac{r}{4L}] \cap K' \neq \emptyset$ and $[t+r-\frac{r}{4L},t+r] \cap K' \neq \emptyset$, because otherwise we would have $\lambda([t,t+r] \setminus K') \geq \frac{r}{2L}$. This implies $\max(K) - \min(K) \geq r-\frac{r}{2L} \geq \frac{r}{2}$, and of course we have a fortiori that $\lambda([\min(K),\max(K)] \setminus K) < \frac{r}{4L}$. Then set $x := \gamma(\min(K))$ and $y := \gamma(\max(K))$. We have a compact $K \subset \RR$ and a 1-Lipschitz map $\gamma: K \to M$ such that $\gamma(\min(K)) = x$, $\gamma(\max(K)) = y$, and $\lambda([\min(K),\max(K)] \setminus K) < \frac{r}{4L}$, showing $d_{ur}(x,y) < \frac{r}{4L}$. But also, the facts that $\max(K) - \min(K) \geq \frac{r}{2}$ and $\gamma^{-1}$ is $L$-Lipschitz imply $d(x,y) \geq \frac{r}{2L}$. Hence, $d_{ur}(x,y) < \frac{r}{4L} < \frac{r}{2L} \leq d(x,y)$, and $q$ is not an isometry.
\end{proof}

\begin{definition} \label{def:Malpha}
Let $(M,d)$ be a metric space. We recursively define a transfinite sequence of metric spaces $(M_{ur}^{(\alpha)},d_{ur}^{(\alpha)})$ and 1-Lipschitz surjections $q_\alpha: M \to M_{ur}^{(\alpha)}$. First, define $(M_{ur}^{(0)},d_{ur}^{(0)})$ to be $(M,d)$ and $q_0$ to be the identity map on $M$. Next, fix an ordinal $\alpha > 0$ and suppose the definition has been made for all $\alpha' < \alpha$. If $\alpha$ is a successor, we define $(M_{ur}^{(\alpha)},d_{ur}^{(\alpha)}) := ((M_{ur}^{(\alpha-1)})_{ur},(d_{ur}^{(\alpha-1)})_{ur})$ and $q_\alpha := q \circ q_{\alpha-1}$ where $q=q^{M_{ur}^{(\alpha-1)}}$ is the 1-Lipschitz surjection $M_{ur}^{(\alpha-1)} \to (M_{ur}^{(\alpha-1)})_{ur}$. If $\alpha$ is a limit ordinal, we define a pseudometric $d_{ur}^{(\alpha)}$ on $M$ by
$$d_{ur}^{(\alpha)}(x,y) := \inf_{\alpha' < \alpha} d_{ur}^{(\alpha')}(q_{\alpha'}(x),q_{\alpha'}(y)).$$
After identifying any points $x,y$ with $d_{ur}^{(\alpha)}(x,y) = 0$, we obtain a metric space $(M_{ur}^{(\alpha)},d_{ur}^{(\alpha)})$ and a canonical 1-Lipschitz surjection $q_\alpha: M \to M_{ur}^{(\alpha)}$. When we wish to emphasize the domain $M$, we write $q^M_\alpha$. Whenever $N$ is a second metric space and $f: M \to N$ is 1-Lipschitz, we get induced 1-Lipschitz maps $f_{ur}^{(\alpha)}: M_{ur}^{(\alpha)} \to N_{ur}^{(\alpha)}$ defined by $f_{ur}^{(\alpha)}(q^M_\alpha(x)) = q^N_\alpha(f(x))$. Well-definedness can be verified by transfinite induction, as well as the functorial property $(f \circ g)_{ur}^{(\alpha)} = f_{ur}^{(\alpha)} \circ g_{ur}^{(\alpha)}$.
\end{definition}

Let us now see that this iterative process always stabilizes after at most countably many steps when $M$ is separable. 

\begin{proposition} \label{prop:stabilize}
For any separable 
metric space $M$, there exists a countable ordinal $\alpha_M$ such that $q: M_{ur}^{(\alpha_M)} \to (M_{ur}^{(\alpha_M)})_{ur}$ is an isometry.
\end{proposition}

\begin{proof}
Let $(M,d)$ be a separable metric space and $D\subset M$ be a countable dense subset. Fix $(x,y) \in D \times D$. Then we get a nonincreasing map $f: \omega_1 \to [0,d(x,y)]$ given by $f(\alpha) := d_{ur}^{(\alpha)}(q_\alpha(x),q_\alpha(y))$. It is easy to see that there is $\alpha_{(x,y)}<\omega_1$ such that $f(\alpha) = f(\alpha_{(x,y)})$ for all $\alpha \in [\alpha_{(x,y)},\omega_1)$. Then we set
$$\alpha_M := \sup\set{ \alpha_{(x,y)} \,:\, (x,y) \in D \times D}$$
and note that $\alpha_M$ is countable since $D \times D$ and each $\alpha_{(x,y)}$ are countable. Observe that, for all $(x,y) \in D \times D$,
$$d_{ur}^{(\alpha_M+1)}(q_{\alpha_M+1}(x),q_{\alpha_M+1}(y)) = d_{ur}^{(\alpha_M)}(q_{\alpha_M}(x),q_{\alpha_M}(y))$$
and thus $q: M_{ur}^{(\alpha_M)} \to (M_{ur}^{(\alpha_M)})_{ur}$ is an isometry restricted to $q_{\alpha_M}(D)$. By density and continuity, $q$ must be an isometry on all of $M_{ur}^{(\alpha_M)}$.
\end{proof}

When $M$ is 1-rectifiable, $\alpha_M \leq 1$. It may happen in general that $\alpha_M > 1$, and in fact we believe $\alpha_M$ can be an arbitrarily large countable ordinal for $M$ compact. See Examples \ref{ex:rectifiable} and \ref{ex:bta}.

\begin{proposition} \label{thm:universalp1u}
Let $M$ be a separable metric space. Then $M_{ur}^{(\omega_1)}$ is purely 1-unrectifiable and satisfies the following universal property: whenever $N$ is a purely 1-unrectifiable metric space and $f: M \to N$ is a 1-Lipschitz map, there exists a unique 1-Lipschitz map $\tilde{f}: M_{ur}^{(\omega_1)} \to N$ such that $f = \tilde{f} \circ q^M_{\omega_1}$.
\end{proposition}

\begin{proof}
That $M_{ur}^{(\omega_1)}$ is purely 1-unrectifiable follows from Propositions~\ref{thm:isop1u} and \ref{prop:stabilize} and transfinite induction. Now let $f: M \to N$ be a 1-Lipschitz map to a purely 1-unrectifiable metric space $N$. Then we get an induced 1-Lipschitz map $f_{ur}^{(\omega_1)}: M_{ur}^{(\omega_1)} \to N_{ur}^{(\omega_1)}$. Since $N$ is purely 1-unrectifiable, $q^N_{\omega_1}: N \to N_{ur}^{(\omega_1)}$ is an isometry by Proposition~\ref{thm:isop1u}. Then $\tilde{f} := (q_{\omega_1}^N)^{-1} \circ f_{ur}^{(\omega_1)}$ satisfies $f = \tilde{f} \circ q^M_{\omega_1}$. Uniqueness follows from the surjectivity of $q_{\omega_1}^M$.
\end{proof}

By \cite[Corollary 8.13]{Weaver2}, $\lip(M_{\LL})$ separates points uniformly when $M$ is compact, and thus $M_{\LL}$ is purely 1-unrectifiable by Theorem~\ref{ThmA}. Hence the canonical 1-Lipschitz surjection $\pi: M \to M_{\LL}$ induces a 1-Lipschitz map $\tilde{\pi}: M_{ur}^{(\omega_1)} \to M_{\LL}$ by Proposition~\ref{thm:universalp1u}. The next theorem is the main one of this section.

\begin{theorem} \label{thm:ML=Mur}
For every compact metric space $M$, the map $\tilde{\pi}: M_{ur}^{(\omega_1)} \to M_{\LL}$ is an isometry.
\end{theorem}

Before proving the theorem, we need to make a small observation. For every metric space $M$, the map $\pi: M \to M_\LL$ satisfies the following universal property: whenever $N$ is a metric space for which $\lip(N)$ separates points uniformly with separation constant 1 and $f: M \to N$ is 1-Lipschitz, there exists a unique 1-Lipschitz map $\tilde{f}: M_\LL \to N$ such that $f = \tilde{f} \circ \pi$. This observation can be proven directly from the definitions and the fact that $g \circ f \in \lip(M)$ whenever $f: M \to N$ is Lipschitz and $g \in \lip(N)$.

\begin{proof}
Let $M$ be a compact metric space, so that $M_{ur}^{(\omega_1)}$ is also compact because $q_{\omega_1}: M \to M_{ur}^{(\omega_1)}$ is a Lipschitz surjection. Then by Proposition~\ref{thm:universalp1u} and Theorem~\ref{ThmA}, $\lip(M_{ur}^{(\omega_1)})$ separates points uniformly. Thus, by the universal property of $\pi: M \to M_{\LL}$, there exists a unique 1-Lipschitz map $\tilde{q}_{\omega_1}: M_{\LL} \to M_{ur}^{(\omega_1)}$ such that $\tilde{q}_{\omega_1} \circ \pi = q_{\omega_1}$. 

\noindent\begin{minipage}{0.6\linewidth}
Since $q_{\omega_1}$ is surjective, so is $\tilde{q}_{\omega_1}$. Similarly $\tilde{\pi} \circ q_{\omega_1} = \pi$ and $\pi$ is surjective, hence so is $\tilde{\pi}$. We get $\tilde{q}_{\omega_1} \circ \tilde{\pi} \circ q_{\omega_1} = q_{\omega_1}$ and, since all these maps are surjective, $\tilde{q}_{\omega_1}$ and $\tilde{\pi}$ are inverses. Then since both are 1-Lipschitz, both are isometries.
\end{minipage}
\begin{minipage}{0.4\linewidth}
$$
\xymatrix{
M \ar[d]_{\pi} \ar[rr]^{q_{\omega_1}} && M_{ur}^{(\omega_1)} \ar@<-0.5ex>[lld]_{\tilde{\pi}} \\
M_{\LL} \ar@<-0.5ex>[rru]_{\tilde{q}_{\omega_1}}
}
$$
\end{minipage}
\end{proof}

\begin{corollary} \label{cor:Whitney}
A compact metric space $M$ fails to be 1-critical if and only if $M_{ur}^{(\omega_1)}$ is a single point. 
\end{corollary}

\subsection{1-critical bounded turning arcs}
\label{subsec:boundedturning}
The object of this subsection is the proof of Theorem \ref{ThmF} (through Theorem \ref{thm:bttree}). We begin by recalling some standard definitions.

A metric space that is homeomorphic to a nonempty, compact interval is called a \textit{metric arc}. We always assume metric arcs are endowed with an order inherited from a homeomorphism to an interval. Although there are always two such orderings (except if $M$ is a single point), the choice is inconsequential.

A metric arc $(M,d)$ satisfying
$$d(x,y) \vee d(y,z) \leq d(x,z)$$
whenever $x \leq y \leq z$ is a \textit{1-bounded turning arc}. Equivalently, $\diam([x,z]) = d(x,z)$ whenever $x \leq z$. A metric space that is bi-Lipschitz equivalent to a 1-bounded turning arc is called a \textit{bounded turning arc}. This is equivalent to the existence of a constant $C < \infty$ such that $\diam([x,z]) \leq Cd(x,z)$ for every subarc $[x,z]$.

\begin{remark}
Whenever $M$ is a 1-bounded turning arc and $A \subset M$, there exists an interval (equivalently, subarc) $I \supset A$ with $\diam(I) = \diam(A)$. We will use this fact implicitly when dealing with coverings and Hausdorff content of subsets of $M$.

Bounded turning arcs were characterized by Meyer (\cite[Corollary 1.2]{Meyer_2011}) as precisely those metric spaces $(M,d)$ for which there exists a homeomorphism $f: [0,1] \to M$ and $H < \infty$ such that $|x-y| \leq |x-z|$ implies $d(f(x),f(y)) \leq H d(f(x),f(z))$; such a homeomorphism is called a weak quasisymmetry.
\end{remark}

The next proposition will be used frequently, often without reference, throughout this subsection.

\begin{proposition} \label{prop:qmonotone}
For every 1-bounded turning arc $M$, the space $M_{ur}$ is a 1-bounded turning arc and $q: M \to M_{ur}$ is monotone. Even more, for every ordinal $\alpha$, $M_{ur}^{(\alpha)}$ is a 1-bounded turning arc and $q_\alpha: M \to M_{ur}^{(\alpha)}$ is monotone.
\end{proposition}

\begin{proof}
Let $M$ be a 1-bounded turning arc.
For the first part, we begin by showing $$d_{ur}(x,y) \vee d_{ur}(y,z) \leq d_{ur}(x,z)$$
whenever $x \leq y \leq z \in M$. Let $x \leq y \leq z \in M$. Let $\eps > 0$, and let $K \subset \RR$ be compact and $\gamma: K \to M$ 1-Lipschitz with $\gamma(\min(K)) = x$, $\gamma(\max(K)) = z$, and
$$\lambda([\min(K),\max(K)] \setminus K) < d_{ur}(x,z) + \eps .$$
If $y \in \gamma(K)$, then $\tilde{K} := K \cap [\min(K),\min(\gamma^{-1}(y))]$ and $\tilde{\gamma} := \gamma \restrict_{\tilde{K}}$ witness $d_{ur}(x,y) < d_{ur}(x,z) + \eps$. If $y \notin \gamma(K)$, then there exist $s,t \in K$ such that
$$(s,t) \subset [\min(K),\max(K)] \setminus K$$
and $y \in [\gamma(s),\gamma(t)]$ (or $y \in [\gamma(t),\gamma(s)]$ if $\gamma(t)< \gamma(s)$; assume the former). Then we define $\tilde{K} := (K \cap [\min(K),s]) \cup \{t\}$ and $\tilde{\gamma}:\tilde{K}\to M$ by
$$
\tilde{\gamma}(r):=\begin{cases} y & \text{if }r=t \\ \gamma(r) & \text{otherwise.} \end{cases}
$$
Because of the 1-bounded turning property, $\tilde{\gamma}$ is 1-Lipschitz, and thus $\tilde{K},\tilde{\gamma}$ witness $d_{ur}(x,y) < d_{ur}(x,z) + \eps$. Since $\eps > 0$ was arbitrary, we have $d_{ur}(x,y) \leq d_{ur}(x,z)$ in all cases. The other inequality $d_{ur}(y,z) \leq d_{ur}(x,z)$ follows from the same argument. These inequalities imply that the order on $M_{ur}$ defined by $q(x) \leq q(y)$ if $x \leq y$ or $q(x)=q(y)$ is well-defined. Obviously $q$ is monotone and the metric topology on $M_{ur}$ is compact, connected, and agrees with the order topology. It is straightforward to use these facts (for example, with \cite[Theorem 28.13]{Willard}) to prove that $M_{ur}$ is a 1-bounded turning arc.

The proof of the second part is by transfinite induction. The base case is trivial because $q_0$ is the identity map. Let $\alpha > 0$ be an ordinal and suppose the proposition holds for all $\alpha' < \alpha$. The case where $\alpha$ is a successor follows immediately from the first part, so assume that $\alpha$ is a limit ordinal. Once we show
$$d_{ur}^{(\alpha)}(x,y) \vee d_{ur}^{(\alpha)}(y,z) \leq d_{ur}^{(\alpha)}(x,z)$$
whenever $x \leq y \leq z \in M$, the same argument from the first part implies $M_{ur}^{(\alpha)}$ is a 1-bounded turning arc and $q_\alpha$ is monotone. Let $x \leq y \leq z \in M$ and $\eps > 0$. Choose $\alpha' < \alpha$ large enough so that $d_{ur}^{(\alpha')}(x,z)< d_{ur}^{(\alpha)}(x,z) + \eps$. By the inductive hypothesis, $d_{ur}^{(\alpha')}(x,y) \vee d_{ur}^{(\alpha')}(y,z) \leq d_{ur}^{(\alpha')}(x,z)$. Together with the inequality from the previous sentence we get $d_{ur}^{(\alpha)}(x,y) \vee d_{ur}^{(\alpha)}(y,z) < d_{ur}^{(\alpha)}(x,z) + \eps$. Since $\eps > 0$ was arbitrary, the desired inequality follows.
\end{proof}

We shall now work our way towards Theorem \ref{ThmF} through a series of lemmas describing the relation between $\HH^1$-$\sigma$-finite sets in $M$ and $M_{ur}^{(\alpha)}$ in cases of increasing coverage.

\begin{lemma} \label{lem:inversediam}
Let $M$ be a 1-bounded turning arc. For every $r > 0$ and $\Sigma \subset M_{ur}$, if $\diam(\Sigma) < r$, then there exist subsets $S,A\subset M$ such that $q^{-1}(\Sigma) \subset S \cup A$, $\HH_r^1(A) < r$, and $\HH^1(S) < \infty$.
\end{lemma}

\begin{proof}
Let $r > 0$ and $\Sigma \subset M_{ur}$ with $\diam(\Sigma) < r$. Since $(M_{ur},d_{ur}^{(1)})$ is a 1-bounded turning arc by Proposition \ref{prop:qmonotone}, there exists an interval $[b_0,b_1] \supset \Sigma$ with $d_{ur}^{(1)}(b_0,b_1) = \diam(\Sigma) < r$. Since $q$ is monotone, $q^{-1}([b_0,b_1]) = [c_0,c_1]$ for some $c_0,c_1 \in M$ with $q(c_0) = b_0$ and $q(c_1)=b_1$, and by definition of $(M_{ur},d_{ur}^{(1)})$ and $q$, it holds that $d_{ur}(c_0,c_1) = d_{ur}^{(1)}(b_0,b_1) < r$. Then it suffices to cover $[c_0,c_1]$ with sets $S,A \subset M$ satisfying $\HH_r^1(A) < r$ and $\HH^1(S) < \infty$.

By definition of $d_{ur}$, there exist a compact $K \subset \RR$ and a 1-Lipschitz map $\gamma: K \to M$ such that $\gamma(\min(K)) = c_0$, $\gamma(\max(K)) = c_1$, and $\lambda([\min(K),\max(K)] \setminus K) < r$. The set $[\min(K),\max(K)] \setminus K$ is a countable disjoint union of intervals $\{(x_i,y_i)\}_i$. The intervals $\{[\gamma(x_i),\gamma(y_i)]\}_i$ (understood to be $[\gamma(y_i),\gamma(x_i)]$ if $\gamma(y_i) \leq \gamma(x_i)$) together with $\gamma(K)$ cover $[c_0,c_1]$. Then $A := \bigcup_i [\gamma(x_i),\gamma(y_i)]$ and $S := \gamma(K)$ satisfy the required properties.
\end{proof}

\begin{lemma} \label{lem:inverseH}
Let $M$ be a 1-bounded turning arc. For every $r > 0$, $\delta \in (0,\infty]$, and $\Sigma \subset M_{ur}$, if $\HH_\delta^1(\Sigma) < r$, then there exist subsets $S,A \subset M$ such that $q^{-1}(\Sigma) \subset S \cup A$, $\HH_\delta^1(A) < r$, and $S$ is $\HH^1$-$\sigma$-finite.
\end{lemma}

\begin{proof}
Let $r > 0$, $\delta \in (0,\infty]$, and $\Sigma \subset M_{ur}$ with $\HH_\delta^1(\Sigma) < r$. Choose $\eps > 0$ such that $\HH_\delta^1(\Sigma) + \eps < r$. Cover $\Sigma$ with countably many sets $\{\Sigma_i\}_i$ such that $\diam(\Sigma_i) < \delta$ and $\sum_i (\diam(\Sigma_i) + 2^{-i}\eps) < r$. By Lemma \ref{lem:inversediam}, there exist, for each $i \geq 1$, subsets $S_i, A_i \subset M$ such that $q^{-1}(\Sigma_i) \subset S_i \cup A_i$, $\HH_\delta^1(A_i) < \diam(\Sigma_i) + 2^{-i}\eps$, and $\HH^1(S_i) < \infty$. Then $S := \bigcup_i S_i$ and $A := \bigcup_i A_i$ satisfy the required properties.
\end{proof}

\begin{lemma} \label{lem:inversesigmafinite}
Let $M$ be a 1-bounded turning arc. For every subset $\Sigma \subset M_{ur}$, if $\Sigma$ is $\HH^1$-$\sigma$-finite, then $q^{-1}(\Sigma) \subset M$ is $\HH^1$-$\sigma$-finite.
\end{lemma}

\begin{proof}
It suffices to assume $\HH^1(\Sigma) < \infty$. Let $k \in \NN$ be arbitrary. By Lemma \ref{lem:inverseH}, there exist subsets $S_k,A_k \subset M$ such that $q^{-1}(\Sigma) \subset S_k \cup A_k$, $\HH_{1/k}^1(A_k) < \HH^1(\Sigma) + 1$, and $S_k$ is $\HH^1$-$\sigma$-finite. Set $S := \bigcup_{k=1}^\infty S_k$ and $A := \bigcap_{k=1}^\infty A_k$. Then $q^{-1}(\Sigma) \subset S \cup A$,
$$\HH^1(A) = \sup_k \HH_{1/k}^1(A) \leq \sup_k \HH_{1/k}^1(A_k) \leq \HH^1(\Sigma) + 1 < \infty,$$
and $S$ is $\HH^1$-$\sigma$-finite. This proves that $q^{-1}(\Sigma)$ is $\HH^1$-$\sigma$-finite.
\end{proof}

\begin{lemma} \label{lem:ordinal}
For any ordinal $\alpha$, 1-bounded turning arc $M$ and subsets $\Sigma,B \subset M_{ur}^{(\alpha)}$ such that $\Sigma$ is $\HH^1$-$\sigma$-finite, there exist subsets $S,A \subset M$ such that $q_\alpha^{-1}(\Sigma \cup B) \subset S \cup A$, $\HH_\infty^1(A) \leq  \HH_\infty^1(B)$ and $S$ is $\HH^1$-$\sigma$-finite.

In particular, $q_\alpha^{-1}(\Sigma)$ is $\HH^1$-$\sigma$-finite whenever $\Sigma$ is $\HH^1$-$\sigma$-finite.
\end{lemma}

\begin{proof}
The proof is by transfinite induction. The base case is tautological since $q_0$ is the identity map. Let $\alpha$ be an ordinal and suppose the lemma holds for all $\alpha' < \alpha$. The case where $\alpha$ is a successor follows immediately from the inductive hypothesis and Lemmas \ref{lem:inverseH} and \ref{lem:inversesigmafinite}, so assume $\alpha$ is a limit ordinal. Let $M$ be a 1-bounded turning arc, $r>0$, and $\Sigma,B \subset M_{ur}^{(\alpha)}$ such that $\Sigma$ is $\HH^1$-$\sigma$-finite. As in Lemma \ref{lem:inversesigmafinite}, it suffices to assume $\HH^1(\Sigma) < \infty$. Let $k \in \NN$ be arbitrary. Cover $B$ with countably many intervals $\{I^k_j\}_{j=1}^\infty$ such that $\sum_j \diam(I^k_j) \leq \HH_\infty^1(B) + 2^{-k}$, and similarly cover $\Sigma$ with countably many intervals $\{J^k_j\}_{j=1}^\infty$ such that $\diam(J^k_j) \leq 2^{-k}$ and $\sum_j \diam(J^k_j) \leq \HH^1(\Sigma) + 1$. Fix $j \in \NN$. Choose $\alpha_j < \alpha$ large enough so that
\begin{align*}
\diam((q_\alpha^{\alpha_j})^{-1}(I^k_j)) &\leq \diam(I^k_j) + 2^{-j-k}, \\
\diam((q_\alpha^{\alpha_j})^{-1}(J^k_j)) &\leq \diam(J^k_j) + 2^{-j-k},
\end{align*}
where $q^{\alpha_j}_{\alpha}: M_{ur}^{(\alpha_j)} \to M_{ur}^{(\alpha)}$ denotes the canonical 1-Lipschitz map. This choice is possible because, by the 1-bounded turning condition, the diameter of any interval is determined by only one distance (the distance between the endpoints) and not infinitely many distances.

By the inductive hypothesis,
\begin{align*}
q_\alpha^{-1}(I^k_j) &= q_{\alpha_j}^{-1}((q_\alpha^{\alpha_j})^{-1}(I^k_j)) \subset C^k_j \cup D^k_j \\
q_\alpha^{-1}(J^k_j) &= q_{\alpha_j}^{-1}((q_\alpha^{\alpha_j})^{-1}(J^k_j)) \subset E^k_j \cup F^k_j
\end{align*}
where
\begin{align*}
\HH_\infty^1(C^k_j) &\leq \diam(I^k_j) + 2^{-j-k} \\
\HH_\infty^1(E^k_j) &\leq \diam(J^k_j) + 2^{-j-k} \leq 2^{-k+1}
\end{align*}
and $D^k_j$, $F^k_j$ are $\HH^1$-$\sigma$-finite. Note that this implies $\HH_{2^{-k+1}}^1(E^k_j) \leq \diam(J^k_j) + 2^{-j-k}$. Set $C^k := \bigcup_j C_j^k$, $D^k := \bigcup_j D_j^k$, $E^k := \bigcup_j E_j^k$, and $F^k := \bigcup_j F_j^k$, so that
\begin{itemize}[itemsep=3pt]
\item $q_\alpha^{-1}(B) \subset C^k \cup D^k$ and $q_\alpha^{-1}(\Sigma) \subset E^k \cup F^k$,
\item $\HH_{\infty}^1(C^k) \leq \sum_j \pare{\diam(I^k_j) + 2^{-j-k}} \leq \HH_\infty^1(B) + 2^{-k+1}$,
\item $\HH_{2^{-k+1}}^1(E^k) \leq \sum_j \pare{\diam(J^k_j) + 2^{-j-k}} \leq \HH^1(\Sigma) + 2 < \infty$, and
\item $D^k$ and $F^k$ are $\HH^1$-$\sigma$-finite.
\end{itemize}
Set $C := \bigcap_k C^k$, $D := \bigcup_k D^k$, $E := \bigcap_k E^k$, and $F := \bigcup_k F^k$ so that
\begin{itemize}[itemsep=3pt]
\item $q_{\alpha}^{-1}(\Sigma \cup B) \subset C \cup D \cup E \cup F$,
\item $\HH_\infty^1(C) \leq \inf_k \HH_\infty^1(C^k) \leq \HH_\infty^1(B)$,
\item $\HH^1(E) = \sup_k \HH_{2^{-k+1}}^1(E) \leq \sup_k \HH_{2^{-k+1}}^1(E^k) \leq \HH^1(\Sigma) + 2 < \infty$, and
\item $D$ and $F$ are $\HH^1$-$\sigma$-finite.
\end{itemize}
Then $A := C$ and $S := D \cup E \cup F$ satisfy the required properties. This finishes the inductive step and the proof.
\end{proof}

The previous lemma is already enough to prove Theorem \ref{ThmF} in the particular case of 1-bounded turning arcs.

\begin{theorem} \label{thm:btarc}
Let $(M,d)$ be a 1-bounded turning arc with endpoints $x,y$. Then $$d_{\LL}(\pi(x),\pi(y)) = d_{ur}^{(\omega_1)}(q_{\omega_1}(x),q_{\omega_1}(y)) = \inf\{\HH^1_\infty(A): M \setminus A \text{ is } \HH^1\text{-}\sigma\text{-finite}\}.$$
\end{theorem}

\begin{proof}
The first equality follows from Theorem~\ref{thm:ML=Mur}. Next, observe that applying Lemma \ref{lem:ordinal} with $\alpha = \omega_1$, $\Sigma = \emptyset$, and $B = M_{ur}^{(\omega_1)}$ gives us
\begin{align*}
\inf\{\HH^1_\infty(A): M \setminus A \text{ is } \HH^1\text{-}\sigma\text{-finite}\} \leq \HH^1_\infty(M_{ur}^{(\omega_1)}) &\leq \diam(M_{ur}^{(\omega_1)}) \\
&= d_{ur}^{(\omega_1)}(q_{\omega_1}(x),q_{\omega_1}(y)) .
\end{align*}
Finally, we show the reverse inequality. Let $f \in \ball{\lip(M)}$ and $A \subset M$ such that $M \setminus A$ is $\HH^1$-$\sigma$-finite. Since $M$ is connected, $|f(y)-f(x)| \leq \diam(f(M)) = \HH^1_\infty(f(M))$. Then
\begin{align*}
|f(y)-f(x)| \leq \HH^1_\infty(f(M)) &\leq \HH^1_\infty(f(A)) + \HH^1_\infty(f(M \setminus A)) \\
&\leq \HH^1_\infty(A) + \HH^1_\infty(f(M \setminus A)) = \HH^1_\infty(A)
\end{align*}
where the last inequality follows from the fact that $f$ is 1-Lipschitz and the last equality follows from the fact that $f$ is locally flat and $M \setminus A$ is $\HH^1$-$\sigma$-finite. Since $A$ and $f$ were arbitrary we get
\begin{align*}
d_{\LL}(\pi(x),\pi(y)) &= \sup\,\{ |f(y)-f(x)| \,:\, f \in \ball{\lip(M)} \} \\
&\leq \inf\,\{\HH^1_\infty(A): M \setminus A \text{ is } \HH^1\text{-}\sigma\text{-finite}\} . \qedhere
\end{align*}
\end{proof}

It is now straightforward to extend Theorem \ref{thm:btarc} to the more general class of 1-bounded turning trees, defined below.

\begin{definition} \label{def:bttree}
A compact metric space $M$ is a \textit{1-bounded turning tree} if every pair of points $x,y \in M$ are joined by a unique arc in $M$, and this arc is 1-bounded turning. We will denote that arc by $[x,y]$. A metric space is a \textit{bounded turning tree} if it is bi-Lipschitz equivalent to a 1-bounded turning tree.
\end{definition}

\begin{example}
A class of metric spaces called \textit{quasiconformal trees} was recently studied by Bonk and Meyer \cite{BM_2020} in the context of quasisymmetric uniformization (see also \cite{Kinneberg} and \cite{DV_2020}). Quasiconformal trees are precisely those bounded turning trees $T$ that are \textit{doubling} - meaning there exists $N \in \NN$ such that for any $r > 0$ and $x \in T$, there are $x_1, \dots x_N \in T$ with $B_r(x) \subset \bigcup_{i=1}^N B_{r/2}(x_i)$. The most well-known examples of quasiconformal trees are Julia sets of polynomials. See \cite{BM_2020} for details and more information.
\end{example}

\begin{theorem} \label{thm:bttree}
Let $(M,d)$ be a 1-bounded turning tree. Then for all $x,y \in M$,
$$d_{\LL}(\pi(x),\pi(y)) = d_{ur}^{(\omega_1)}(q_{\omega_1}(x),q_{\omega_1}(y)) = \inf\{\HH^1_\infty(A): [x,y] \setminus A \text{ is } \HH^1\text{-}\sigma\text{-finite}\}.$$
In particular, a bounded turning tree fails to be 1-critical if and only if each of its subarcs is $\HH^1$-$\sigma$-finite.
\end{theorem}

\begin{proof}
Let $x,y \in M$. The first equality is again Theorem \ref{thm:ML=Mur}. The proof of the inequality
$$d_{\LL}(\pi(x),\pi(y)) \leq \inf\{\HH^1_\infty(A): [x,y] \setminus A \text{ is } \HH^1\text{-}\sigma\text{-finite}\}$$
is just the same as the proof of the analogous inequality from Theorem \ref{thm:btarc}. It remains to show the reverse inequality.
To that end, it is enough to check that $[x,y]$ is a 1-Lipschitz retract of $M$. Indeed, this implies that the restriction map $\ball{\lip(M)} \to \ball{\lip([x,y])}$ is a surjection and, together with Theorem \ref{thm:btarc}, this proves
$$d_{\LL}(\pi(x),\pi(y)) \geq \inf\{\HH^1_\infty(A): [x,y] \setminus A \text{ is } \HH^1\text{-}\sigma\text{-finite}\}.$$

Fix an order for $[x,y]$ with $x\leq y$, and define a mapping $g:M\to M$ by
$$ g(z)=\sup([x,y]\cap [x,z]) . $$
Clearly $g$ is a retraction onto $[x,y]$. We claim that
$$
[x,y]\cap [g(z),z] = \set{g(z)}
$$
for each $z\in M$: suppose $w\in [x,y]\cap [g(z),z]$, then $w\in [g(z),z]\subset [x,z]$ and thus $w\leq g(z)$ by the definition of $g(z)$. On the other hand, $g(z)\in [x,z]$ and $w\in [g(z),z]$ clearly imply $g(z)\in [x,w]$ and hence $w\geq g(z)$, proving our claim. It now follows that, for $w,z\in M$, we have either $g(w)=g(z)$ or $[g(w),g(z)]= [w,z]\cap [x,y]$, and in the latter case
$$
d(g(w),g(z)) = \diam([g(w),g(z)]) \leq \diam([w,z]) = d(w,z)
$$
i.e. $g$ is 1-Lipschitz. This finishes the proof of the first part.

For the second part, observe that the property of being 1-critical and the property of being $\HH^1$-$\sigma$-finite are each preserved under bi-Lipschitz maps, and thus it suffices to prove the equivalence when $M$ is a 1-bounded turning tree. If every arc $[x,y] \subset M$ is $\HH^1$-$\sigma$-finite, then we may take $A = \emptyset$ in the first part and get $d_\LL(\pi(x),\pi(y)) = 0$ for every $x,y \in M$, meaning $M$ is not 1-critical. Conversely, if $M$ is not 1-critical, then $d_\LL(\pi(x),\pi(y)) = 0$ for every $x,y \in M$, and so by the first part we may find a sequence $\{A_n\}_{n=1}^\infty$ of subsets of $[x,y]$ such that $[x,y] \setminus A_n$ is $\HH^1$-$\sigma$-finite and $\HH_\infty^1(A_n) \leq \frac{1}{n}$ for every $n \in \NN$. This implies $\HH_\infty^1(\bigcap_{n=1}^\infty A_n) = 0$, and so
$$[x,y] = \pare{\bigcup_{n=1}^\infty [x,y] \setminus A_n} \cup \pare{\bigcap_{n=1}^\infty A_n}$$
is $\HH^1$-$\sigma$-finite.
\end{proof}

\subsection{Curve-flat Lipschitz functions}
In this final subsection, we define curve-flat Lipschitz functions and sketch examples of bounded turning arcs $M$ whose index $\alpha_M$ from Proposition \ref{prop:stabilize} may be arbitrarily large.

\begin{definition}[Curve-Flat Lipschitz Functions]
\label{def:curveflat}
Let $(M,d)$ be a metric space. A Lipschitz function $f: M \to \RR$ is \textit{curve-flat} if for any compact $K \subset \RR$ and Lipschitz $\gamma: K \to M$, the metric differential of the composite $f \circ \gamma$ vanishes at $\lambda$-almost every point in $K$, i.e.
$$\lim_{\substack{y \to x \\ y\in K}}\frac{|f(\gamma(x))-f(\gamma(y))|}{|x-y|} = 0$$
for $\lambda$-almost every $x \in K$.
By the area formula (Theorem \ref{thm:areaformula}), this is equivalent to $\lambda(f(\gamma(K))) = 0$ for every $K \subset \RR$ compact and $\gamma: K \to M$ Lipschitz.
Let $\lip_\Gamma(M)$ denote the set of all curve-flat Lipschitz functions on $M$ and $\ball{\lip_\Gamma(M)}$ the set of curve-flat 1-Lipschitz functions. Define a pseudometric $d_\Gamma$ on $M$ by
$$d_\Gamma(x,y) := \sup\set{|f(y)-f(x)|: f \in \ball{\lip_\Gamma(M)}}.$$
\end{definition}

Curve-flat Lipschitz functions are intimately connected to the pseudometric $d_{ur}$ by the following proposition.

\begin{proposition} \label{prop:dur=dG}
For any metric space $(M,d)$ and $f \in \ball{\Lip(M_{ur})}$, $f \circ q \in \ball{\lip_\Gamma(M)}$. Consequently, $d_{ur} = d_\Gamma$.
\end{proposition}

\begin{proof}
Let $f \in \ball{\Lip(M_{ur})}$. Clearly $g:=f\circ q\in\Lip(M)$ is $1$-Lipschitz, so it remains to prove that $g$ is curve-flat. Let $K\subset\RR$ be compact and $\gamma:K\to M$ be $1$-Lipschitz, and fix $a\leq b\in K$. Then
\begin{align*}
\abs{g\circ\gamma(a)-g\circ\gamma(b)} &= \abs{f\circ q\circ\gamma(a)-f\circ q\circ\gamma(b)} \\
&\leq \lipnorm{f}\cdot d_{ur}(q(\gamma(a)),q(\gamma(b))) \\
&\leq d_{ur}(\gamma(a),\gamma(b)) \\
&\leq \lambda([a,b]\setminus K)
\end{align*}
where we apply the definition of $d_{ur}$ to the restriction of $\gamma$ to $K\cap [a,b]$. By Lebesgue's density theorem, we have $\lambda([a,b]\setminus K)/(b-a)\to 0$ as $b\to a$ for $\lambda$-almost every $a\in K$ (or as $a\to b$ for $\lambda$-almost every $b\in K$). Thus $g$ is curve-flat.

The first statement immediately implies $d_{ur} \leq d_\Gamma$. For the reverse inequality, fix $x,y\in M$, $f\in\ball{\lip_\Gamma(M)}$, a compact $K\subset\RR$ with $a=\min(K)$, $b=\max(K)$, and a 1-Lipschitz map $\gamma:K\to M$ such that $\gamma(a)=x$, $\gamma(b)=y$. Then $f\circ\gamma\in\ball{\Lip(K)}$ is flat at $\lambda$-almost every point in $K$. Extend $f\circ\gamma$ to a function $h\in\ball{\Lip([a,b])}$. By Rademacher's theorem $h$ is differentiable $\lambda$-a.e. in $[a,b]$ with $\norm{h'}_\infty\leq 1$ and $h'=0$ $\lambda$-a.e. in $K$. Hence, by Lebesgue's fundamental theorem of calculus
$$
\abs{f(y)-f(x)} = \abs{h(b)-h(a)} = \abs{\int_a^b h'\,d\lambda} = \abs{\int_{[a,b]\setminus K} h'\,d\lambda} \leq \lambda([a,b]\setminus K) .
$$
Taking the supremum over $f$ in the left hand side and the infimum over $K$ in the right hand side yields $d_\Gamma(x,y)\leq d_{ur}(x,y)$.
\end{proof}

We can use Proposition \ref{prop:dur=dG} to prove that $q$ collapses every curve fragment in $M$ down to an $\HH^1$-null subset of $M_{ur}$.

\begin{proposition} \label{prop:collapse}
For every metric space $M$, compact $K \subset \RR$, and Lipschitz $\gamma: K \to M$, $\HH^1(q(\gamma(K))) = 0$.
\end{proposition}

\begin{proof}
Suppose the proposition is false. Then by Lemma \ref{lem:Kirchheim}, we can find $K \subset \RR$ compact with $\lambda(K) > 0$ and $\gamma: K \to M$ Lipschitz such that $q \circ \gamma: K \to M_{ur}$ is a bi-Lipschitz embedding. Then $(q \circ \gamma)^{-1}: q(\gamma(K)) \to \RR$ is a bi-Lipschitz embedding, and we let $g: M_{ur} \to \RR$ be any McShane-Whitney extension. By Proposition \ref{prop:dur=dG}, $g \circ q: M \to \RR$ is curve-flat Lipschitz, but $\lambda(g(q(\gamma(K)))) = \lambda(K)  > 0$, a contradiction.
\end{proof}

We will use the remainder of this subsection to estimate the index $\alpha_M$ of some example spaces $M$.

\begin{example}[1-Rectifiable Metric Spaces] \label{ex:rectifiable}
Recall that a metric space is 1-rectifiable if it is the union of countably many curve fragments and an $\HH^1$-null set. Proposition~\ref{prop:collapse} implies $\HH^1(M_{ur}) = 0$, and hence $M_{ur}$ is purely 1-unrectifiable, whenever $M$ is 1-rectifiable. In this case, $\alpha_M = 0$ if $\HH^1(M) = 0$ and $\alpha_M = 1$ if $\HH^1(M) > 0$.
\end{example}

Every 1-rectifiable metric space is $\HH^1$-$\sigma$-finite, and so Example \ref{ex:rectifiable} may tempt one to believe $\alpha_M \leq 1$ whenever $M$ is $\HH^1$-$\sigma$-finite. Our final example shows that this is not the case.

\begin{example}[Bounded Turning Arcs] \label{ex:bta}
We will sketch the construction of an $\HH^1$-$\sigma$-finite 1-bounded turning arc $M$ with $\diam(M) = \diam(M_{ur}) = 1$. Theorem \ref{thm:btarc} implies $M_{ur}^{(\omega_1)}$ is a single point and thus $\alpha_M > 1$. We believe the construction can be iterated to make $\alpha_M$ an arbitrarily large countable ordinal, but we leave those details to the interested reader.

Let $\CC \subset [0,1]$ be the standard middle thirds Cantor set and $\beta := \log_3(2)$ the Hausdorff dimension of $\CC$. The Cantor function $f: \CC^\beta \to [0,1]$ is monotone, surjective, and 1-Lipschitz (\cite[Proposition 10.1]{Cantor}), where $\CC^\beta$ denotes the snowflake space. Let $(M,d)$ be the metric space obtained by ``filling in the gaps" of $\CC^\beta$ with geodesics. Precisely, consider the collection of all doubletons $\{x < y\} \subset \CC^\beta$ such that $[x,y] \cap \CC = \{x,y\}$, then form the disjoint union $\CC^\beta \sqcup \bigsqcup_{\{x<y\}} [0,d(x,y)]$, identify each $x$ with its copy 0 and each $y$ with $d(x,y)$, and equip the resulting quotient space $M$ with the largest metric $d$ such that the inclusions $\CC^\beta, [0,d(x,y)] \hookrightarrow (M,d)$ are isometric embeddings. The space $M$ is a 1-bounded turning arc with $\diam(M) = 1$. The Cantor function $f$ extends to a 1-Lipschitz map $f: M \to [0,1]$ that is constant on each geodesic. Observe that $f$ is curve-flat because $\CC^\beta$ is purely 1-unrectifiable and $f$ is constant on each of the countably many geodesics. This gives us $d_{ur}(x,y) = d_\Gamma(x,y) \geq |f(x)-f(y)|$. In particular, $\diam(M_{ur}) = 1$. However, $\HH^1(\CC^\beta) = \HH^\beta(\CC) < \infty$, implying that $M$ is $\HH^1$-$\sigma$-finite. Thus, $M_{ur}^{(\omega_1)} = M_\LL$ is a single point by Theorem~\ref{thm:btarc}.
\end{example}


\section*{Acknowledgments}

The authors wish to thank David Bate, Marek C\'uth, Gilles Godefroy, Miguel Mart\'in and Nik Weaver for their suggestions on the topic and presentation of the paper, as well as the anonymous referee for their very careful reading of our manuscript.

R. J. Aliaga was partially supported by the Spanish Ministry of Economy, Industry and Competitiveness under Grant MTM2017-83262-C2-2-P.
C. Petitjean and A. Proch\'azka were partially supported by the French ANR project No. ANR-20-CE40-0006.


\end{document}